\documentclass[a4paper]{scrartcl}

\usepackage{multirow}
\usepackage[utf8]{inputenc}
\usepackage[T1]{fontenc}
\usepackage[english]{babel}
\usepackage{amsmath,amssymb,amsfonts,siunitx,commath, amsthm}
\usepackage{kpfonts}
\usepackage{tikz,url}
\usepackage{geometry}
\usepackage{mathtools}
\mathtoolsset{showonlyrefs}
\usepackage{caption, subcaption}
\usepackage{float}
\usepackage{algorithm}
\usepackage{algpseudocode}
\usepackage{algorithmicx}
\usepackage{pgfplotstable}  
\usepackage[colorlinks,pdfusetitle,
]{hyperref}
\usepackage{enumerate}
\usepackage{listings}
\usepackage[normalem]{ulem}
\usepackage{cancel}
\usepackage{dsfont}
\usepackage{graphicx}
\usepackage{comment}
\usepackage{pgfplots}
\usetikzlibrary{matrix,positioning}
\usepgfplotslibrary{groupplots}
\pgfplotsset{compat=1.17}

\newcommand{\MOT}{\ensuremath{\text{MOT}}}
\newcommand{\MOTe}{\ensuremath{\text{MOT}_{\eta}}}
\newcommand{\Ker}{\boldsymbol{\mathcal{K}}}
\newcommand{\bOmega}{\boldsymbol{\Omega}}
\newcommand{\bx}{{\boldsymbol{x}}}
\newcommand{\bI}{{\boldsymbol{I}}}
\newcommand{\bi}{{\boldsymbol{i}}}

\newcommand{\bn}{{\boldsymbol{n}}}

\newcommand{\bC}{{\boldsymbol{C}}}
\newcommand{\bF}{{\boldsymbol{F}}}

\makeatletter

\newcommand*{\rom}[1]{\expandafter\@slowromancap\romannumeral #1@}
\makeatother

\newcommand{\R}{\mathbb{R}}

\newcommand{\C}{\mathbb{C}}

\newcommand{\N}{\mathbb{N}}

\newcommand{\tT}{\intercal}
\newcommand{\im}{\mathrm{i}}

\newcommand{\uproman}[1]{\uppercase\expandafter{\romannumeral#1}}

\renewcommand\norm[1]{\left\lVert#1\right\rVert}

\theoremstyle{plain}
\newtheorem{lemma}{Lemma}[section]
\newtheorem{theorem}[lemma]{Theorem}

\newtheorem{proposition}[lemma]{Proposition}

\theoremstyle{definition}

\algdef{SE}[DOWHILE]{Do}{doWhile}{\algorithmicdo}[1]{\algorithmicwhile\ #1}%

\makeatletter
\pgfplotsset{
	sticky options/.store in=\pgfplots@group@stickyoptions,
	sticky options={}
}

\def\pgfplots@group@nextplot[#1]{%
	\ifnum\pgfplots@group@current@plot=\pgfplots@group@totalplots\relax
	\pgfplotswarning{groupplots/too many plots}{\the\pgfplots@group@current@plot}{\pgfplots@group@totalplots}\pgfeov
	\else
	\ifnum0<\pgfplots@group@current@plot\relax
	\endpgfplots@environment@opt
	\fi
	
	\begingroup
	\pgfqkeys{/pgfplots}{#1}
	\pgfmath@smuggleone\pgfplots@group@stickyoptions
	\endgroup
	\pgfplots@group@increment@numbers
	\pgfplots@group@environment@create{#1, \pgfplots@group@stickyoptions}
	\fi
}
\makeatother

\makeatletter
\newcommand\gpsubtitle[1]{(\@alph{\pgfplots@group@current@plot}) #1}
\makeatother

\usepackage[nodayofweek,level]{datetime}
\begin{document}
	\title{Accelerating the Sinkhorn algorithm for sparse multi-marginal optimal transport by fast Fourier transforms}
	
	\author{
		Fatima Antarou Ba\footnotemark[1]\\ {\footnotesize\href{mailto:fatimaba@math.tu-berlin.de}{fatimaba@math.tu-berlin.de}}
		\and 
		Michael Quellmalz\footnotemark[1]\\ {\footnotesize\href{mailto:quellmalz@math.tu-berlin.de}{quellmalz@math.tu-berlin.de}}
	}
	
		\date{August 26, 2022}
	\maketitle

	\footnotetext[1]{
		TU Berlin,
		Institute of Mathematics,
		Stra{\ss}e des 17. Juni 136,
		D-10587 Berlin, Germany.
	} 
	\begin{abstract}
		We consider the numerical solution of the discrete multi-marginal optimal transport (MOT) by means of the Sinkhorn algorithm. In general, the Sinkhorn algorithm suffers from the curse of dimensionality  with respect to the number of marginals. If the MOT cost function decouples according to a tree or circle, its complexity is linear in the number of marginal measures. In this case, we speed up the convolution with the radial kernel required in the Sinkhorn algorithm by non-uniform fast Fourier methods. Each step of the proposed accelerated Sinkhorn algorithm with a tree-structured cost function has a complexity of $\mathcal O(K N)$ instead of the classical $\mathcal O(K N^2)$ for straightforward matrix--vector operations, where $K$ is the number of marginals and each marginal measure is supported on at most $N$ points. In case of a circle-structured cost function, the complexity improves from $\mathcal O(K N^3)$ to $\mathcal O(K N^2)$. This is confirmed by numerical experiments.
	\end{abstract}
	
	\section{Introduction}
	
	The optimal transport (OT) problem is an optimization problem that deals with the search for an optimal map (plan) that moves masses between two or more measures at low cost \cite{peyre2020computational,Villani2009}.
	OT appears in a wide range of applications such as  image and signal processing \cite{beier2021unbalanced, WassersteinBarycentric2016, PeyreWasserseinloss2016, Thorpe2017, Vogt2018}, economics \cite{carlier2014numerical,Galichon2016}, finance \cite{Dolinsky2014, RobustDolinsky2014}, 
	and physics \cite{Frisch2002ARO,haasler2021multimarginal}.
	The OT problem was first introduced in 1781 by 
	Monge. His objective was to find a map between two probability measures $\mu^1, \mu^2$ on $\R^d$ that transports $\mu^1$ to $\mu^2$ with minimal cost,
	where the cost function describes the cost of transporting mass between two points in $\R^d$. However, 
	such maps do not always exist, so that
	Kantorovich \cite{Ka58} relaxed the problem in 1942 by looking for a transport plan with two prescribed marginals $\mu^1$ and $\mu^2$ that minimizes a certain cost functional.
	
	Several authors have generalized the formulation to multi-marginal optimal transport (MOT)~\cite{lin2020complexity,pass2014multiagent,pass2014multimarginal}, where more than two marginal measures are given.
	For given probability measures $\mu^k$ on $\Omega^k\subset\R^d$, $k=1,\dots,K$,
	an optimal transport plan $\pi$ is defined as a solution of the MOT problem
	\begin{equation}\label{eq:MOT_cont}
		\min_{\pi \in \Pi(\mu^1, \ldots, \mu^k)} \int_{\Omega^1\times\dots\times\Omega^K} c(x^1, \ldots, x^K) \,\mathrm d\pi(x^1, \ldots, x^K) 
	\end{equation}
	where 
	$\Pi(\mu^1, \cdots, \mu^K)$ is the convex set of all joint probability measures $\pi$ whose marginals are $\mu^k$,
	and $c\colon \Omega^1\times\dots\times\Omega^K \to \R$ is the cost function. 
	
	Since the numerical computation of a transport plan is difficult in general, a regularization term such as the entropy \cite{MR3617575,haasler2021multimarginal}, Kullback-Leibler divergence \cite{neumayer2020optimal}, general $f$-divergence~\cite{terjek2021optimal} or $L^2$-regularization~\cite{pmlr-v84-blondel18a,lorenz2019quadratically} can be added to make the problem strictly convex. 
	Different approaches such as the Sinkhorn algorithm~\cite{haasler2021multimarginal, peyre2020computational}, stochastic gradient descent \cite{aude2016stochastic}, the Gauss-Seidel method \cite{lorenz2019quadratically}, or the proximal splitting~\cite{Ammari2014} have been used to iteratively determine a minimizing sequence of the MOT problem. 
	
	However, the problem suffers from the curse of dimensionality as the complexity grows exponentially with the number $K$ of marginal measures. 
	One way to circumvent this lies in incorporating additional sparsity assumptions into the model. Polynomial-time algorithms to solve certain sparse MOT problems have been studied in~\cite{AltschulerBoix-AdseraPolynomial, BenamouCarlierCuturi2015, haasler2021multimarginal}. We will assume that the cost function decouples according to a graph, where the nodes correspond to the marginals and 
	the cost function is the sum of functions that depend only on two variables which correspond to two marginals connected by an edge of the graph.
	For example, the circle-structured Euclidean cost function reads as
	$$
	c(x^1,\dots,x^K)
	=
	\norm{x^1-x^2}_2^2 
	+ \dots 
	+ \norm{x^{K-1}-x^K}_2^2
	+ \norm{x^{K}-x^1}_2^2,
	\quad x^1,\dots,x^K\in\R^d.
	$$
	MOT problems with graph-structured cost functions with constant treewidth can be solved with the Sinkhorn algorithm in polynomial time \cite{KollerFriedman2009}.
	In \cite{AltschulerBoix-AdseraPolynomial}, 
	polynomial-time algortihms were presented for the MOT problem and its regularized counterpart
	for the cases of a graph structure, a set-optimization structure, as well as a low-rank and sparsely structured cost function.
	Another sparsity assumption lies on the transport plan to be thinly spread, which is e.g.\ the case for the $L^2$-regularized problem \cite{pmlr-v84-blondel18a,lorenz2019quadratically}.

	\paragraph{\textbf{Our contributions.}}
	
	In the present paper, we study the discrete, entropy-regularized $\MOT$ problem with a tree-structured \cite{beier2021unbalanced,haasler2021multimarginal} or a circle-structured cost function, where all measures are supported on a finite number of points (atoms) in $\R^d$.
	Then the computational time of the Sinkhorn algorithm \cite{alaya2020screening, NIPS2013_af21d0c9, Sinkhorn1967ConcerningNM}, 
	which iteratively determines a sequence converging to the solution,
	depends linearly on the number $K$ of input measures.
	If the numbers of atoms is large, however, the Sinkhorn algorithm still requires considerable computational time and memory,
	which mainly comes from computing a discrete convolution, i.e.\ a matrix--vector product with a kernel matrix.
	
	This is significantly improved by Fourier-based fast summation methods \cite{post02,postni04}.
	The key idea is the approximation of the kernel function by a Fourier series, which enables the application of the non-uniform fast Fourier transform (NFFT).
	Although such fast summation methods are frequently used in different applications such as electrostatic particle interaction \cite{NePiPo13}, tomography \cite{HiQu15}, image segmentation \cite{AlPoStVo18}
	and, very recently, also OT with two marginals \cite{fastfouriersinkhorn2022},
	they were not utilized for MOT so far.
	Furthermore, a method for accelerating the Sinkhorn algorithm for Wasserstein barycenters via computing the convolution with a different kernel, namely the heat kernel, was discussed in \cite{Solomon2015}.
	
	Our main contribution is the combination of the fast summation method with the sparsity of the tree- or circle-structured cost function in the MOT problem for accelerating the Sinkhorn algorithm.
	Each iteration step has a complexity of 
	$\mathcal O(K N)$ for a tree-
	and $\mathcal O(K N^2)$ for a circle-structured cost function,
	compared to $\mathcal O(K N^2)$ and $\mathcal O(K N^3)$, respectively, with the straightforward matrix--vector operations,
	where $N$ is an upper bound of the number of atoms for each of the $K$ marginal measures.
	Our numerical tests with both tree- and circle-structured cost functions confirm a considerable acceleration while the accuracy stays almost the same.
	{
		A different acceleration of the Sinkhorn algorithm via low-rank approximation for tree-structured cost yields the same asymptotic complexity \cite{StrKre22}.
		We note that MOT problems with tree-structured cost functions are used for the computation of Wasserstein barycenters \cite{WassersteinBarycentric2016}, 
		and with circle-structured cost for computing Euler flows \cite{BenamouCarlierNenna20217}.
	}
	
	\paragraph{\textbf{Outline of the paper.}}
	Section~\ref{sec:Notation} introduces the notation.
	In section~\ref{sec:MOT_general_problem_setting}, we focus on the discrete MOT problem with squared Euclidean norm cost functions and the numerical solution of the corresponding entropy-regularized problem by the Sinkhorn algorithm. 
	We investigate sparsely structured cost functions 
	that decouple according to a tree or circle in section~\ref{sec:sparse_MOT}. Then the complexity of the Sinkhorn algorithm depends only linearly on $K$. 
	Section~\ref{sec:NFFT-Sinkhorn} describes a fast summation method for further accelerating the Sinkhorn algorithm.
	Finally, in section~\ref{sec:Numerical_examples},
	we verify the practical performance by applying it to generalized Euler flows and for finding generalized Wasserstein barycenters. We compare the computational time of the proposed Sinkhorn algorithm based on the NFFT with the algorithm based on direct matrix multiplication. 
	\section{Notation}
	\label{sec:Notation}
	Let $K\in\N$ and $\bn = (n_1,\dots,n_k) \in \N^K$.
	We set $[K] \coloneqq \set{1,\cdots, K}$
	and consider $K$ finite sets
	$$
	\Omega^k
	=
	\left\{\bx^k_{i_k}  : i_k\in[n_k]\right\}
	\subset \R^d,
	\quad k\in[K],
	$$
	consisting of points $\bx^k_{i_k}\in\R^d$, which are called \emph{atoms}. We set  
	$\bOmega
	\coloneqq
	\Omega^1\times\dots\times\Omega^K$.
	Additionally, we define 
	the index set
	\begin{equation}
		\boldsymbol{I}\coloneqq 
		\left\{ \bi = (i_1,\dots,i_K) : i_k\in[n_k], k\in [K] \right\}
	\end{equation}
	and the set of $K$-dimensional matrices (tensors) $$\R^{\bi}\coloneqq \R^{i_1}\times\dots\times \R^{i_K}, \quad \bi \in \bI.$$
	
	Let $\mathbb{P}(\Omega^k)$ denote the set of probability measures on $\Omega^k$.
	In this paper, we consider $K$ discrete probability measures, also called \emph{marginal measures}, $\mu^k \in \mathbb{P}(\Omega^k)$ given by 
	\begin{equation} \label{eq:mu_discrete}
		\mu^k = \sum_{i_k=1}^{n_k} \mu_{i_k}^k \delta_{\bx_{i_k}^k},
		\quad k\in[K],
	\end{equation}
	where the probabilities  satisfy 
	$$
	\mu_{i_k}^k \geq 0,\quad \sum_{i_k=1}^{n_k}\mu_{i_k}^k=1
	\quad\text{for all } i_k\in[n_k],\ k\in[K],
	$$ 
	and, for all $A \subset \Omega^k$, the Dirac measure is given by
	\begin{equation}
		\delta_{x^k_{i_k}}(A) \coloneqq \begin{cases}
			1 & \text{if }  x^k_{i_k}\in A,\\
			0 & \text{otherwise}.
		\end{cases}
	\end{equation}

	For $\boldsymbol G,\boldsymbol  H\in \R^{\bn}$,
	we denote their component-wise (Hadamard) product by
	\begin{equation}
		\boldsymbol G\odot \boldsymbol H \coloneqq \left( G_{\bi} H_{\bi}\right)_{\bi\in\bI} \in \R^{\bn},
	\end{equation}
	and
	similarly their component-wise division by $\oslash$,
	as well as the Frobenius inner product
	\begin{equation}
		\langle \boldsymbol G, \boldsymbol H\rangle_\mathrm{F} \coloneqq
		\sum_{\bi\in\bI}  G_{\bi} H_{\bi} \in \R.
	\end{equation}
	The tensor product (Kronecker product) of $\boldsymbol u,\boldsymbol v \in \R^m$ is denoted by $\boldsymbol u\otimes \boldsymbol v \in \R^{m\times m}$. 
	Analogues can be defined for tensors of different size.
	
	\section{Multi-marginal optimal transport}\label{sec:MOT_general_problem_setting}
	In the following, we consider the discrete \emph{multi-marginal optimal transport} (MOT) between $K$ marginal measures $\mu^k\in\mathbb{P}(\Omega^k)$, $k\in[K]$.
	We define the set of \emph{admissible transport plans} by
	\begin{equation} \label{eq:feasible_discrete_plan}
		\Pi(\mu^1, \ldots, \mu^k) \coloneqq \set{ \boldsymbol{\Pi} \in \R^\bn_{\ge0} :P_k(\boldsymbol{\Pi})
			= \mu^k \text{ for all } k\in [K]},
	\end{equation}
	where the $k$-th marginal projection is defined as 
	\begin{equation} \label{eq:Pk}
		P_k(\boldsymbol{\Pi})\coloneqq
		\sum_{\ell\in [K]\setminus{\set{k}}}\sum_{i_{\ell} \in [n_{\ell}]}\boldsymbol{\Pi}_{i_1,\cdots, i_K} \in \R^{n_k}.
	\end{equation}
	For a \emph{cost function} $c\colon\bOmega \rightarrow \R_{\ge0}$
	and 
	the samples $\bx_\bi=\left(x^1_{i_1}, \ldots, x^K_{i_K}\right)$, $\bi\in\bI$,
	we define the respective \emph{cost matrix} 
	\begin{equation}\label{eq:C}
		\boldsymbol{C} \coloneqq [\boldsymbol{C}_{\bi}]_{\bi\in\bI} = [c(\bx_{\bi})]_{\bi\in\bI}=[c(\bx^1_{i_1}, \cdots, \bx^K_{i_K})]_{(i_1,\cdots,i_K)\in \bI} \in \R_{\ge0}^{\bn}.
	\end{equation}
	The discrete MOT problem reads
	\begin{equation} \label{eq:MOT}
		\min_{\boldsymbol{\Pi} \in \Pi(\mu^1, \ldots, \mu^k)}\langle \boldsymbol{\Pi}, \bC\rangle_\mathrm{F},
	\end{equation}
	whose solution $\boldsymbol{\Pi}\in \R^\bn_{\ge0}$ is called \emph{optimal plan}.
	
	\subsection{Entropy regularization}\label{subsec:entropy_reg}
	
	As the MOT problem \eqref{eq:MOT} is numerically unfeasible,
	we consider for $\eta>0$ the entropy-regularized multi-marginal optimal transport ($\MOTe$) problem
	\begin{equation}\label{eq:MOTe}
		\min_{\boldsymbol{\Pi} \in \Pi(\mu^1, \ldots, \mu^K)} \langle \boldsymbol{\Pi}, \boldsymbol{C}\rangle_\mathrm{F}  + \eta \langle \boldsymbol{\Pi}, \log \boldsymbol{\Pi}-\boldsymbol{1}_\bn \rangle_\mathrm{F} = \min_{\boldsymbol{\Pi} \in \Pi(\mu^1, \ldots, \mu^k)} \sum_{\bi \in \boldsymbol{I}} \boldsymbol{\Pi}_{\bi}\boldsymbol{C}_{\bi} + \eta \sum_{\bi\in\bI} \boldsymbol{\Pi}_{\bi} \left(\log \boldsymbol{\Pi}_{\bi}-1\right),
	\end{equation}
	which is a convex optimization problem.
	It is possible to numerically deduce the optimal transport plan $\hat{\boldsymbol \Pi}$ of \eqref{eq:MOTe} from the solution of the corresponding Lagrangian dual problem.
	The following theorem is a special case of \cite{beier2021unbalanced} for a constant entropy function.
	\begin{theorem}\label{dual_form_MOT_reg}
		The Lagrangian dual formulation of the discrete $\MOTe$ problem \eqref{eq:MOTe} states 
		\begin{equation}\label{eq:dual_problem}
			\sup_{\substack{\phi^k \in \R_{\ge0}^{n_k}, k \in [K]}}
			\mathcal{S}\left(\phi^1, \cdots, \phi^K\right)
			\coloneqq
			\sup_{\substack{\phi^k \in \R_{\ge0}^{n_k}, k \in [K]}}
			\eta\sum_{k \in [K]}\sum_{j \in [n_k]} \mu^k_{j}\log\phi^k_{j} -\eta \sum_{\bi\in\bI} \Ker_{\bi} \boldsymbol{\Phi}_{\bi},
		\end{equation}
		where the kernel matrix $\Ker \in \R^\bn$ is defined by
		\begin{equation}
			\Ker_{\bi} \coloneqq \exp\left(-\frac{\boldsymbol{C}_{\bi}}{\eta}\right),
			\quad \bi\in\bI,
		\end{equation}
		and the dual tensor $\boldsymbol{\Phi}= \bigotimes_{k=1}^K \phi^k$ by
		\begin{equation}
			\boldsymbol{\Phi}_{\bi} \coloneqq \prod_{k=1}^K \phi_{i_k}^k,
			\quad \bi\in\bI.
		\end{equation}
		The functional $\mathcal{S}\colon \R_{\ge0}^{\bn} \rightarrow \R$ is called the Sinkhorn function.
	\end{theorem}
	The solutions of the dual and the primal problem are generally not equal. 
	The solution of \eqref{eq:dual_problem} is generally a lower bound to the solution of the primal problem~\eqref{eq:MOTe}. 
	Equality holds if the cost function $c$ is lower semi-continuous, i.e.,
	$
	\liminf c(\boldsymbol{x}) \geq c(\boldsymbol{x_0})
	$
	as $\boldsymbol{x}\rightarrow \boldsymbol{x_0}$
	for every $\boldsymbol{x_0} \in \R^{\bn}$, or Borel measurable and bounded \cite{beiglbock2012}. This is obviously the case for the squared Euclidean norm cost function
	\begin{equation} \label{eq:norm}
		c\colon \R^{\bn}\rightarrow [0,\infty),\quad
		c(\bx)=\sum_{\substack{k_1,k_2 \in [K]\\ k_1\neq k_2}}\norm{x^{k_1}_{i_{k_1}}-x^{k_2}_{i_{k_2}}}_2^2
	\end{equation}
	that we study here.
	
	\begin{proposition}[\cite{elvander2019multimarginal}]
		An optimal plan of the $\MOTe$ problem \eqref{eq:MOTe} is given by
		\begin{equation}
			\boldsymbol{\hat{\Pi}} = \Ker\odot\hat{\boldsymbol{\Phi}},
		\end{equation}
		where $\hat{\boldsymbol{\Phi}}=\bigotimes_{k\in[K]}\hat{\phi}^k$ and $\hat{\phi}^k,$ $k\in [K]$ are the optimal solutions of dual problem \eqref{eq:dual_problem}.
	\end{proposition}
	A sequence converging to the optimal dual vectors $\hat{\phi}^k,$ $k\in [K]$ in \eqref{eq:dual_problem} can be iteratively determined by the \emph{Sinkhorn algorithm} \cite{haasler2021multimarginal,Marino2020} 
	presented in Algorithm~\ref{alg:MOT-Sinkhorn-K},
	where we note that line $4$ is obtained by deriving the Sinkhorn function~$\mathcal S$ with respect to $\phi^k,$ $k\in [K].$  
	The complexity of the algorithm mainly comes from the computation of the marginal $P_k(\Ker \odot \boldsymbol{\Phi})$,
	where the projection $P_k$ is defined in \eqref{eq:Pk}.
	In general, the number of operations depends exponentially on $K$.
	\begin{algorithm}[!ht]
		\begin{algorithmic}[1]
			\State \textbf{Input: } Initial values $(\phi^k)^{(0)}\in\R^{n_k},$ $k\in [K],$ regularization parameter $\eta>0,$ threshold $\delta>0$
			\State Set $r\leftarrow 0$
			\Do
			\For {$k=1,\ldots, K$}
			\State Compute $\left(\boldsymbol{\tilde{\Phi}}\right)^{(r+1)}_k\coloneqq\bigotimes_{\ell \in [k-1]}\left(\phi^{\ell}\right)^{(r+1)}\otimes \bigotimes_{{\ell}\in[K]\setminus[k-1]}\left(\phi^{\ell}\right)^{(r)}$
			\State Compute dual vectors
			\begin{align}
				&\left(\phi^k\right)^{(r+1)} \coloneqq \left(\mu^k \odot \left(\phi^k\right)^{(r)}\right)\oslash P_k\left(\Ker\odot\left(\boldsymbol{\tilde{\Phi}}\right)^{(r+1)}_k\right)
			\end{align}
			\State Increment $r\leftarrow r+1$
			\EndFor
			\doWhile{$\abs{\mathcal{S}\left(\left(\phi^1\right)^{(r)},\dots,\left(\phi^K\right)^{(r)}\right)-\mathcal{S}\left(\left(\phi^1\right)^{(r-1)},\dots,\left(\phi^K\right)^{(r-1)}\right)} \geq \delta$}
			\State \Return optimal plan $\hat{\boldsymbol{\Pi}}= \Ker\odot \hat{\boldsymbol{\Phi}}$ and $\hat{\boldsymbol{\Phi}} = \bigotimes_{k\in[K]} \left(\phi^k\right)^{(r+1)}$
		\end{algorithmic}
		\caption{Sinkhorn iterations for the $\MOTe$ problem}
		\label{alg:MOT-Sinkhorn-K}
	\end{algorithm}

	\section{Sparse cost functions}\label{sec:sparse_MOT}
	
	In this section, we take a look at sparsely structured cost functions,
	for which the Sinkhorn algorithm becomes much faster and we overcome the curse of dimensionality. 
	Let $G=\left(\mathcal{V}, \mathcal{E}\right)$ be an undirected graph with vertices $\mathcal{V}$ and edges $\mathcal{E}$.  
	We say that the $\MOTe$ problem has the graph~$G$ structure if $\mathcal V = [K]$ and the cost matrix \eqref{eq:C} decouples according to
	\begin{equation} \label{eq:cost_sparse}
		\boldsymbol{C}_{\bi} 
		=
		\sum_{\{k_1,k_2\}\in \mathcal{E}} 
		\big\|{x^{k_1}_{i_{k_1}}- x^{k_2}_{i_{k_2}}}\big\|_2^2
		,\quad \bi = (i_1,\dots,i_K)\in \bI.
	\end{equation}
	This implies that the kernel $\Ker\in \R_{\ge0}^\bn$ satisfies
	\begin{equation} \label{eq:K-sparse}
		\Ker_{\bi} 
		= \exp\left(-\frac{\boldsymbol{C}_{\bi}}{\eta}\right) 
		= \prod_{\{k_1,k_2\}\in\mathcal E}\Ker_{i_{k_1}, i_{k_2}}^{(k_1,k_2)}
		,\quad \bi\in\bI,
	\end{equation}
	where the kernel matrix $\Ker^{(k_1,k_2)} \in \R_{\ge0}^{n_{k_1}\times n_{k_2}}$ for $\{k_1,k_2\}\in\mathcal E$ is given by
	\begin{equation} \label{eq:K-matrix}
		\Ker_{i_{k_1}, i_{k_2}}^{(k_1,k_2)} \coloneqq
		\exp\left(-\frac{1}{\eta}{\big\|{\bx^{k_1}_{i_{k_1}}-\bx^{k_2}_{i_{k_2}}}\big\|_2^2 }\right) . 
	\end{equation}
	We use the indices $k \in \mathcal{V}=[K]$ to identify the marginal measures $\mu^k$ in the rest of the paper. 
	
	The discrete, dual formulation \eqref{eq:dual_problem} of the $\MOTe$ problem has the same form independently of the structure of the graph $G$, only the marginals $P_k\left(\Ker \odot \boldsymbol{\Phi}\right)$ differ.
	If the graph $G$ is complete, i.e., each two of the vertices in $\mathcal{V}$ are connected with an edge, 
	then the computational complexity of the Sinkhorn Algorithm \ref{alg:MOT-Sinkhorn-K} depends exponentially on the number $K$ of marginal measures (vertices). 
	For larger values of $K$, it is practically impossible to numerically compute an optimal plan to the $\MOTe$ problem. 
	{We} consider two sparsity assumptions of $\MOTe$ problem,
	each of them yielding that the Sinkhorn algorithm has a linear complexity in the number $K$ of nodes. It was shown in \cite{AltschulerBoix-AdseraPolynomial} that $\MOTe$ problems with graphically structured cost functions of constant tree-width can be implemented in polynomial time. This is the case for the tree and circle, whose tree-widths are $1$ and $2$, respectively.
	In the following, we give an explicit scheme to efficiently compute the Sinkhorn iterations for tree and circle structure.

	\subsection{Tree structure}
	\label{sec:tree}
	
	We consider the $\MOTe$ problem with the structure of a
	\emph{tree} $\left(\mathcal{V}, \mathcal{E}\right)$,
	which is a connected and circle-free graph with $\abs{\mathcal{E}}=\abs{\mathcal{V}}-1.$ 
	We define the (non-empty) \emph{neighbour} set $\mathcal{N}_k$ of $k\in\mathcal{V}$ as the set of all nodes $\ell\in \mathcal{V}$ such that $\{k,\ell\}\in \mathcal{E}.$ Furthermore we denote by $\mathcal{L}\coloneqq \{k \in \mathcal{V}: |\mathcal N_k|=1\}$ the set of all \emph{leaves} of the tree.
	
	We call  $1\in\mathcal V$ the \emph{root} of the tree.
	For every $k\in\mathcal V$, there is a unique path between $k$ and the root.
	For $k\in\mathcal V\setminus\set{1}$,
	we define the \emph{parent} $p(k)$ as the node in $\mathcal N_k$ such that $p(k)$ lies in the path between $k$ and the root $1$.
	The root has no parent.
	Without loss of generality,
	we assume that $p(k)<k$ holds for all $k\in\mathcal V\setminus\set{1}$.
	We define the set of \emph{children} 
	$\mathcal C_k \coloneqq \{\ell\in \mathcal N_k : \ell > k\}$. Then, we can derive a recursive formula for the $k$-th marginal of the tensor $\Ker\odot\boldsymbol\Phi \in \R^{\bn}.$
	
	\begin{theorem}
		\label{thm:marginal_tree}
		Let $(\mathcal{V}, \mathcal{E})$ be a tree with leaves $\mathcal L$ and $c$ be the $\MOTe$ cost function associated. Let furthermore $k\in \mathcal{V}$ be an arbitrary node. Then the $k$-th marginal of the transport plan $\Ker\odot \boldsymbol{\Phi}$ is given by
		\begin{equation}
			P_k\left(\Ker \odot \boldsymbol{\Phi}\right) = \phi^k \odot \bigodot_{\ell \in \mathcal{N}_k} \alpha^{(k,\ell)},
		\end{equation}
		where the vectors $\alpha^{(k,\ell)}\in\R^{n_k}$ are recursively defined as 
		\begin{equation*}
			\alpha^{(k,\ell)} = 
			\begin{cases}\Ker^{(k,\ell)}\phi^{\ell}  &\text{if } \ell\in \mathcal{L},\\
				\Ker^{(k,\ell)}\left(\phi^{\ell} \odot \bigodot_{t \in \mathcal{N}_{\ell}\setminus\set{k}} \alpha^{(\ell,t)}\right) & \text{otherwise}.
			\end{cases}
		\end{equation*}
	\end{theorem}
	
	A proof of Theorem \ref{thm:marginal_tree} can be found in \cite[Theorem~3.2]{haasler2021multimarginal}.
	The main idea is to split the rooted tree at the node $k$ into $\abs{\mathcal{N}_k}$ subtrees. Therefore, the kernel matrix holds
	\begin{equation}
		\Ker_{\bi}=\prod_{\{k_1,k_2\}\in \mathcal{E}}\Ker^{(k_1, k_2)}_{i_{k_1}, i_{k_2}} =\prod_{\ell\in\mathcal N_k} \prod_{t\in \mathcal{D_{\ell}}}\Ker^{(p(t), t)}_{i_{p(t)}, i_{t}}, 
		\quad \bi\in\bI,
	\end{equation}
	where $\mathcal{D_{\ell}}$ is the set of descendants of $\ell,$ i.e., the nodes $t\in \mathcal{V}$ such that $\ell$ lies in the path between $k$ and $t.$
	Inserting $\Ker_{\bi}$ into {$P_k\left(\Ker \odot \boldsymbol{\Phi}\right)$ and recalling the definition of $P_k$ in \eqref{eq:Pk}} yields the result.
	
	In order to efficiently perform the Sinkhorn algorithm, we compute iteratively for $\ell=K,\dots,2$
	the vectors $\beta_{\ell}\coloneqq \alpha^{(p(\ell),\ell)}\in \R^{n_{p(\ell)}}.$ From Theorem~\ref{thm:marginal_tree}, we obtain
	\begin{align}
		&\beta_{\ell} =
		\begin{cases}
			\Ker^{(p(\ell),\ell)}\phi^{\ell} &\text{if } \ell\in \mathcal{L},\\
			\Ker^{(p(\ell),\ell)}\left(\phi^{\ell}\odot \bigodot_{t \in \mathcal{C}_{\ell}}\beta_t\right)& \text{otherwise}.
		\end{cases}
	\end{align}
	Since we assumed that $p(k)<k$, the computation of $\beta_\ell$ requires only $\beta_t$ for $t>\ell$.
	Similarly, the vectors $\gamma_\ell \coloneqq \alpha^{(\ell,p(\ell))} \in \R^{n_\ell}$ for $\ell>1$ and $\gamma_1 \coloneqq\mathbf{1}_{n_1}$ can be iteratively computed by
	\begin{equation}\label{eq:gamma_tree}
		\gamma_\ell = \left(\Ker^{(p(\ell),\ell)}\right)^{\intercal}\left(\phi^{p(\ell)}\odot \gamma_{p(\ell)}\odot \bigodot_{t\in \mathcal{C}_{p(\ell)}\setminus\set{\ell}} \beta_t\right),
		\quad \ell=2,\dots,K,
	\end{equation}
	where the vectors $\beta_t\in \R^{n_{p(\ell)}}$ are assumed to be known from above. 
	Then the marginals are
	\begin{equation}
		P_k\left(\Ker\odot \boldsymbol{\Phi}\right) =\phi^k \odot \gamma_k \odot \beta^{\odot_{C_k}},
		\quad k\in[K],
	\end{equation}
	where for any subset $\mathcal{U}\subseteq \mathcal{C}_k,$ we define 
	\begin{equation}\label{eq:comp_wise_multi_beta_tree}
		\R^{n_k} \ni
		\beta^{\odot_{\mathcal{U}}}
		\coloneqq
		\begin{cases}
			\mathbf{1}_{n_k} & \text{ if } \mathcal{U}=\emptyset,\\
			\bigodot_{t\in \mathcal{U}}\beta_t & \text{ otherwise.}
		\end{cases}
	\end{equation} 
	The resulting procedure is summarized in Algorithm~\ref{alg:sink-tree}.
	
	\begin{algorithm}[ht]
		\caption{Sinkhorn algorithm for tree structure\label{alg:sink-tree}}
		\begin{algorithmic}[1]
			\State \textbf{Input:} Tree $T(\mathcal V,\mathcal E)$ with leaves $\mathcal L$ and root $1\in\mathcal V=[K]$, initialization $(\phi^k)^{(0)}$, ${k\in [K]},$ parameters $\eta, \delta>0$
			\State Initialize $r\gets0$
			\For{$k=K,\dots,2$}
			\State
			$\displaystyle
			\beta_{k}^{(0)} \coloneqq 
			\begin{cases}\Ker^{(p(k),k)}\left(\phi^k\right)^{(0)} &\text{if } k\in \mathcal{L},\\
				\Ker^{(p(k),k)}\left(\left(\phi^k\right)^{(0)} \odot \left(\beta^{\odot_{\mathcal{C}_k}}\right)^{(0)}\right) &\text{otherwise}
			\end{cases}
			$
			\EndFor
			\Do
			\For{$k=1,\dots,K$}
			\State $\displaystyle
			\gamma_k^{(r)} \coloneqq
			\begin{cases}
				\mathbf{1} &\text{if } k=1,\\ \left(\Ker^{(p(k),k)}\right)^{\intercal}\left(\left(\phi^{p(k)}\right)^{(r+1)}\odot \gamma_{p(k)}^{(r)}\odot \left(\beta^{\odot_{\mathcal{C}_{p(k)}\setminus\set{k}}}\right)^{(r)}\right)
				&\text{otherwise}
			\end{cases}
			$
			\State Compute the dual vector
			$
			\left(\phi^k\right)^{(r+1)} \coloneqq {\mu^k} \oslash \left({\gamma_k^{(r)} \odot \left(\beta^{\odot_{\mathcal{C}_{k}}}\right)^{(r)}}\right)
			$
			\EndFor
			\For{$k=K,\dots,2$}
			\State Compute $\beta_k^{(r+1)}$ according to step 4.
			\EndFor
			\State Set
			$\displaystyle
			\mathcal{S}^{(r)} \coloneqq {\eta}\left(\sum_{k=1}^K\left(\mu^k\right)^{\intercal}\log\left(\phi^k\right)^{(r)}-\left({\mu^1}\oslash{\left(\phi^1\right)^{(r+1)}}\right)^{\intercal}\left(\phi^1\right)^{(r)}\right)
			$
			\State Increment $r\gets r+1$
			\doWhile{$|{\mathcal{S}^{(r)}-\mathcal{S}^{(r-1)}}| \geq \delta$}
			\State \Return optimal plan $\hat{\boldsymbol{\Pi}}= \Ker\odot \hat{\boldsymbol{\Phi}}$, where $\hat{\boldsymbol{\Phi}} = \bigotimes_{k\in [K]} \left(\phi^k\right)^{(r)}$
		\end{algorithmic}
	\end{algorithm}

	\subsection{Circle structure}
	
	We consider the $\MOTe$ problem where the graph $(\mathcal{V}, \mathcal{E})$ is a circle.
	We assume for each $k\in \mathcal{V}=[K]$ that $ \{k,k+1\}\in \mathcal{E}$ is an edge of the circle, where we set $k+1=1$ if $k=K$ and $k-1=K,$ if $k=1.$
	Thus we can define the distance between two nodes $k_1, k_2\in\mathcal V$ as
	\begin{equation}
		d(k_1,k_2)
		\coloneqq
		\begin{cases}
			k_2-k_1 & \text{if } k_2\geq k_1,\\
			K-k_1+k_2 &\text{otherwise.}
		\end{cases}
	\end{equation}
	
	\begin{theorem}\label{prop:marginal_circle}
		Let $(\mathcal{V}, \mathcal{E})$ be a circle and $c$ be the MOT cost function associated. Let furthermore $k\in \mathcal{V}$ be an arbitrary node. The $k$-th marginal of the transport plan $\Ker\odot \boldsymbol{\Phi}$ is given by
		\begin{equation}
			P_k\left(\Ker\odot \boldsymbol{\Phi}\right) = \left(\left(\phi^k\odot\Ker^{(k, k+1)}\right)\odot\left(\phi^{k+1}\odot\alpha^{(k+1, k)}\right)^{\intercal}\right) \boldsymbol{1}_{k+1},
		\end{equation}
		where the matrices $\alpha^{(\ell,t)}\in \R^{n_{\ell}\times n_t}$, $\ell, t\in \mathcal{V}$ recursively satisfy
		\begin{equation} \label{eq:alpha_circle}
			\alpha^{(\ell,t)}= 
			\begin{cases}
				\Ker^{(\ell, t)} &\text{if } \text{d}(\ell,t)=1,\\
				\Ker^{(\ell, \ell+1)}\left( \phi^{\ell+1}\odot\alpha^{(\ell+1,t)}\right) & \text{otherwise,}	
			\end{cases}	
		\end{equation}
		and we set $k+1=1$ if $k=K$  and $k-1=N,$ if $k=1.$
	\end{theorem}
	
	\begin{proof}
		The kernel $\Ker$ can be decomposed as \eqref{eq:K-sparse}.
		Let $k\in[K]$ and $i_k \in [n_k]$.
		It holds then
		\begin{align}
			[P_k\left(\Ker \odot \boldsymbol{\Phi} \right)]_{i_k} 
			&= \sum_{\ell\in \mathcal{V}\setminus{\set{k}}}\sum_{i_{\ell} \in [n_{\ell}]} \Ker_{\bi} \odot \boldsymbol{\Phi}_{\bi}  
			=\phi^k_{i_k}\sum_{\ell\in \mathcal{V}\setminus{\set{k}}}\sum_{i_{\ell} \in [n_{\ell}]}\prod_{\{k_1, k_2\}\in \mathcal{E}} \Ker^{(k_1, k_2)}_{i_{k_1}, i_{k_2}}\prod_{j\in \mathcal{V}\setminus\set{k}} \phi^{j}_{i_j}\\
			&= \phi^k_{i_k}\sum_{i_1}\phi^{1}_{i_{1}}\sum_{i_2}\Ker^{(1, 2)}_{i_1, i_2}\phi^{2}_{i_2}\cdots\sum_{i_{k-2}}\Ker^{(k-3, k-2)}_{i_{k-3}, i_{k-2}}\phi^{k-2}_{i_{k-2}} \sum_{i_{k-1}}\Ker^{(k-2, k-1)}_{i_{k-2}, i_{k-1}}\phi^{k-1}_{i_{k-1}}\Ker^{(k-1, k)}_{i_{k-1}, i_{k}}\\
			&\hspace{0.9cm}\sum_{i_{k+1}}\Ker^{(k, k+1)}_{i_{k}, i_{k+1}}\phi^{k+1}_{i_{k+1}}\cdots \sum_{i_K}\Ker^{(K-1, K)}_{i_{K-1}, i_{K}}\phi^{K}_{i_{K}}\Ker^{(K, 1)}_{i_{K}, i_{1}}\\
			&= \phi^k_{i_k}\sum_{i_{k+1}}\Ker^{(k, k+1)}_{i_{k}, i_{k+1}}\phi^{k+1}_{i_{k+1}}\cdots \sum_{i_K}\Ker^{(K-1, K)}_{i_{K-1}, i_{K}}\phi^{K}_{i_{K}}\sum_{i_1}\Ker^{(K, 1)}_{i_{K}, i_{1}}\phi^{1}_{i_{1}}\sum_{i_2}\Ker^{(1, 2)}_{i_{1}, i_{2}}\phi^{2}_{i_{2}}\cdots\\
			&\hspace{0.9cm}\sum_{i_{k-2}}\Ker^{(k-3, k-2)}_{i_{k-3}, i_{k+2}}\phi^{k-2}_{i_{k-2}}\sum_{i_{k-1}}\Ker^{(k-2, k-1)}_{i_{k-2}, i_{k-1}}\phi^{k-1}_{i_{k-1}} \Ker^{(k-1, k)}_{i_{k-1}, i_{k}}
			.
		\end{align}    
		Setting $\psi^\ell \coloneqq \phi^{(k+\ell-1)\text{mod }K}$ and $\tilde{\Ker}^{(j,\ell+1)}\coloneqq \Ker^{\left((k+\ell-1)\text{mod }K, (k+\ell)\text{mod }K\right)}$ for every $ \ell\in[K]$, 
		we obtain
		\begin{align}
			[P_k\left(\Ker \odot \boldsymbol{\Phi} \right)]_{i_k} 
			&= \phi^k_{i_k}\sum_{j_2}\tilde{\Ker}^{(1, 2)}_{i_k, j_2}\psi^{2}_{j_2}
			\cdots\sum_{j_{K-1}}\tilde{\Ker}^{(K-2, K-1)}_{j_{K-2}, j_{K-1}}\psi^{K-1}_{j_{K-1}} \sum_{j_K}\tilde{\Ker}^{(K-1, K)}_{j_{K-1}, j_{K}}\psi^{K}_{j_{K}}\tilde{\Ker}^{(K, 1)}_{j_{K}, i_k}
			\\&
			= [P_1\left(\tilde{\Ker} \odot \tilde{\boldsymbol{\Phi}} \right)]_{i_k}.
		\end{align}
		We define 
		\begin{equation}
			\tilde{\alpha}^{(K,1)} \coloneqq\tilde{\Ker}^{(K,1)} = \alpha^{(k-1,k)} \in \R^{n_{k-1}\times n_k}
		\end{equation}
		and recursively for all $\ell=K-1,\dots,1$ the matrix $\tilde\alpha^{(\ell,1)}\in\R^{((k+\ell-1)\text{mod }K)\times n_k}$ by
		\begin{equation}
			\tilde{\alpha}^{(\ell,1)}_{j_{l},i_k}\coloneqq \big[\tilde{\Ker}^{(\ell, \ell+1)}\left(\psi^{\ell+1}\odot \tilde{\alpha}^{(\ell+1, 1)}\right)\big]_{j_{l}, i_k} =\sum_{j_{\ell+1}}\tilde{\Ker}^{(\ell, \ell+1)}_{j_{\ell}, j_{\ell+1}}{\psi}^{\ell+1}_{j_{\ell+1}} \tilde{\alpha}^{(\ell+1, 1)}_{j_{\ell+1}, i_k}
			.
		\end{equation}
		Inserting this into the marginal yields
		\begin{align}
			[P_1\left(\tilde{\Ker} \odot \tilde{\boldsymbol{\Phi}} \right)]_{i_k} &= \phi^k_{i_k}\sum_{j_2}\tilde{\Ker}^{(1, 2)}_{i_k, j_2}\psi^{2}_{j_2}
			\cdots\sum_{j_{K-2}}\tilde{\Ker}^{(K-2, K-1)}_{j_{K-2}, j_{K-1}}\psi^{K-1}_{j_{K-1}}
			\sum_{j_{K-1}}\tilde\Ker^{(K-2, K-1)}_{j_{K-2}, j_{K-1}}\psi^{K-1}_{j_{K-1}}\tilde\alpha^{(K-1,1)}_{j_{K-1}, i_k}
			\\
			&= \psi^1_{i_k}\sum_{j_2}\tilde{\Ker}^{(1, 2)}_{i_k, j_2}\psi^{2}_{j_2}\tilde{\alpha}^{(2,1)}_{j_2,i_k} =  \psi^1_{i_k}{\big[\tilde\Ker}^{(1, 2)}_{i_k,j_2}\big]_{j_2\in [n_{k+1}]}^{\intercal}\left(\psi^{2}\odot\big[\tilde{\alpha}^{(2,1)}_{j_2, i_k}\big]_{j_2\in [n_{k+1}]}\right).
		\end{align}
		Henceforth,
		\begin{equation}
			P_1\left(\tilde{\Ker} \odot \tilde{\boldsymbol{\Phi}} \right) =  \left(\left(\psi^1\odot\tilde{\Ker}^{(1, 2)}\right)\odot\left(\psi^{2}\odot\tilde{\alpha}^{(2,1)}\right)^{\intercal}\right) \boldsymbol{1}.
		\end{equation}
		Finally defining $\tilde{\alpha}^{(\ell,1)}=\alpha^{((k+\ell)\text{mod }K,k)}$ yields the assumption.
	\end{proof}
	To efficiently compute the marginal optimal transport 
	as for the tree structure, we choose $k=1$ as starting point and decompose the marginal into two matrices, which can be computed recursively as follows.
	For matrices  $\boldsymbol G,\boldsymbol H \in \R^{n\times m},$ we define the inner product with respect to the second dimension by
	\begin{equation}
		\langle \boldsymbol G, \boldsymbol H\rangle \coloneqq \left(\sum_{j=1}^m  G_{ij} H_{ij}\right)_{i\in [n]} \in \R^n.
	\end{equation}
	\begin{theorem}\label{prop:marg_new_circ}
		Under the assumptions of Theorem~\ref{prop:marginal_circle}, we have
		\begin{equation}
			P_k\left(\Ker \odot \boldsymbol{\Phi}\right) = 
			\begin{cases}
				\phi^1\odot \left\langle \Ker^{(1,2)},\left(\phi^2\odot\alpha^{(2,1)}\right)^{\intercal}\right\rangle, & k=1,\\ 
				\phi^k\odot\left\langle \alpha^{(k,1)}, \left(\phi^1\odot \lambda^{(1,k)}\right)^{\intercal}\right\rangle, &k=2,\dots,K,
			\end{cases}
		\end{equation}
		where $\alpha^{(k,1)}$ is given in \eqref{eq:alpha_circle} and 
		\begin{equation} \label{eq:eta}
			\lambda^{(1,k)} \coloneqq
			\begin{cases}
				\Ker^{(1,2)}, & k=2,\\ \lambda^{(1,k-1)}\left(\phi^{(k-1)}\odot\Ker^{(k-1,k)}\right), & k=3,\dots,K.
			\end{cases}
		\end{equation}
	\end{theorem}
	\begin{proof}
		Let $k\in[K]$ and $i_k\in[n_k]$.
		For $k=1$, the assertion follows from Theorem~\ref{prop:marginal_circle}. For $k>2$, we have
		\begin{align}
			\left[P_k\left(\Ker \odot \boldsymbol{\Phi}\right)\right]_{i_k} &= \phi^k_{i_k}\sum_{i_1}\phi^{1}_{i_{1}}\sum_{i_2}\Ker^{(1, 2)}_{i_1, i_2}\phi^{2}_{i_2}\cdots
			\sum_{i_{k-1}}\Ker^{(k-2, k-1)}_{i_{k-2}, i_{k-1}}\phi^{k-1}_{i_{k-1}}\Ker^{(k-1, k)}_{i_{k-1}, i_{k}}\\
			&\hspace{0.9cm}\sum_{i_{k+1}}\Ker^{(k, k+1)}_{i_{k}, i_{k+1}}\phi^{k+1}_{i_{k+1}}\cdots \sum_{i_K}\Ker^{(K-1, K)}_{i_{K-1}, i_{K}}\phi^{K}_{i_{K}}\Ker^{(K, 1)}_{i_{K}, i_{1}}\\
			&= \phi^k_{i_k}\sum_{i_1}\phi^{1}_{i_{1}}\underbrace{\sum_{i_2}\Ker^{(1, 2)}_{i_1, i_2}\phi^{2}_{i_2}\cdots 
				\sum_{i_{k-1}}\Ker^{(k-2, k-1)}_{i_{k-2}, i_{k-1}}\phi^{k-1}_{i_{k-1}}\Ker^{(k-1, k)}_{i_{k-1}, i_{k}}}_{\mathrm{\uproman{1}}\coloneqq}\alpha^{(k,1)}_{i_k,i_1}.
		\end{align}
		Furthermore the term $\mathrm{\uproman{1}}$ can be rewritten as
		\begin{align}
			\mathrm{\uproman{1}}= \sum_{i_{k-1}}\Ker^{(k-1, k)}_{i_{k-1}, i_{k}}\phi^{k-1}_{i_{k-1}}\sum_{i_{k-2}}\Ker^{(k-2, k-1)}_{i_{k-2}, i_{k-1}}\phi^{k-2}_{i_{k-2}}\cdots \sum_{i_2}\Ker^{(1, 2)}_{i_1, i_2}\phi^{2}_{i_2}\Ker^{(2,3)}_{i_{2}, i_{3}}
		\end{align}
		By \eqref{eq:eta}, we have
		\begin{equation}
			\lambda^{(1,3)}_{i_1,i_3}=\sum_{i_2}\Ker^{(1, 2)}_{i_1, i_2}\phi^{2}_{i_2}\Ker^{(2,3)}_{i_{2}, i_{3}}
			,\quad i_1\in[n_1],\ i_3\in[n_3].
		\end{equation}
		This implies that
		\begin{align}
			\mathrm{\uproman{1}}&= \sum_{i_{k-1}}\Ker^{(k-1, k)}_{i_{k-1}, i_{k}}\phi^{k-1}_{i_{k-1}}\sum_{i_{k-2}}\Ker^{(k-2, k-1)}_{i_{k-2}, i_{k-1}}\phi^{k-2}_{i_{k-2}}\cdots \sum_{i_3}\Ker^{(3,4)}_{i_3,i_4}\phi^{3}_{i_3}\lambda^{(1,3)}_{i_1, i_3},\\
			&=\cdots = \sum_{i_{k-1}}\Ker^{(k-1, k)}_{i_{k-1}, i_{k}}\phi^{k-1}_{i_{k-1}}\lambda^{(1,k-1)}_{i_1, i_{k-1}}= \big[\lambda^{(1,k-1)}\left(\phi^{k-1}\odot\Ker^{(k-1, k)}\right)\big]_{i_1,i_k}
			=\lambda^{(1,k)}_{i_1, i_{k}}.
		\end{align}
		Hence, the hypothesis is true for all $k>2$. The case $k=2$ can be proved with the same procedure.
	\end{proof}
	
	In order to efficiently compute a maximizing sequence of the dual $\MOTe$ problem~\eqref{eq:dual_problem},
	we set for $k\in [K]$ the dual matrices
	\begin{align}
		\beta_k \coloneqq \alpha^{(k,1)} \in \R^{n_k\times n_1},\\
		\gamma_k \coloneqq \lambda^{(k,1)} \in \R^{n_1\times n_k},
	\end{align}
	given in Theorems \ref{prop:marginal_circle} and \ref{prop:marg_new_circ}, respectively.
	The method is shown in Algorithm \ref{alg:sink-circle}.
	\begin{algorithm}[ht]
		\caption{Sinkhorn algorithm for circle structure}\label{alg:sink-circle}
		\begin{algorithmic}[1]
			\State \textbf{Input:} Initialization $(\phi^k)^{(0)},\ {k\in [K]},$ parameters $\eta, \delta>0$
			\State Initialize $r\leftarrow0$
			\For{$k=K,\cdots,2 $}
			\State $\displaystyle
			\left(\beta_k\right)^{(0)} \coloneqq
			\begin{cases}
				\Ker^{(K,1)} &\text{if } k=K,\\
				\Ker^{(k, k+1)}\left(\left(\phi^{k+1}\right)^{(0)}\odot\left(\beta_{k+1}\right)^{(0)}\right) &\text{otherwise}
			\end{cases}
			$
			\EndFor
			\Do
			\State $\displaystyle
			\left(\phi^{1}\right)^{(r+1)}\coloneqq\mu^1\oslash\left\langle \Ker^{(1,2)},\left( \left(\phi^2\right)^{(r)}\odot\left(\beta_2\right)^{(r)}\right)^{\intercal}\right\rangle
			$
			\For{$k=2,\cdots,K$}
			\State 
			$\displaystyle
			\left(\gamma_k\right)^{(r)} \coloneqq
			\begin{cases}
				\Ker^{(1,2)} &\text{if } k=2,\\ 
				\left(\gamma_{k-1}\right)^{(r)} \left(\left(\phi^{k-1}\right)^{(r+1)}\odot \Ker^{(k-1,k)}\right) &\text{otherwise }
			\end{cases}
			$
			\State Compute dual vector
			$\displaystyle
			\left(\phi^{k}\right)^{(r+1)}\coloneqq\mu^k\oslash\left\langle \left(\beta_k\right)^{(r)}, \left(\left(\phi^1\right)^{(r+1)}\odot \left(\gamma_k\right)^{(r)}\right)^{\intercal}\right\rangle
			$
			\EndFor
			\For{$k=K,\cdots,2$}
			\State Compute $\left(\beta_k\right)^{(r+1)}$ according to step $4$
			\EndFor
			\State Compute
			$\displaystyle
			\mathcal{S}^{(r)} \coloneqq {\eta}\left(\sum_{k=1}^K\left(\mu^k\right)^{\intercal}\log\left(\phi^k\right)^{(r)}-\left({\mu^1}\oslash{\left(\phi^1\right)^{(r+1)}}\right)^{\intercal}\left(\phi^1\right)^{(r)}\right)
			$
			\State Increment $r\leftarrow r+1$
			\doWhile{$|{S^{(r)}-S^{(r-1)}}| \geq \delta$}
			\State \Return Optimal plan $\hat{\boldsymbol{\Pi}}= \Ker\odot \hat{\boldsymbol{\Phi}}$, where $\hat{\boldsymbol{\Phi}} = \bigotimes_{k\in[K]} \left(\phi^k\right)^{(r+1)}$
		\end{algorithmic}
	\end{algorithm}
	
	The tree-structured and circle-structured $\MOTe$ problems have both a sparse cost function which considerably improves the computational complexity of the Sinkhorn algorithm. In each iteration step, Algorithm~\ref{alg:sink-tree} requires only ${2(K-1)}$ matrix--vector products, which have a complexity of $\mathcal O(N^2)$ where $N\coloneqq \norm{\bn}_\infty$,
	and Algorithm~\ref{alg:sink-circle} requires ${2(K-1)}$ matrix--matrix products, which have a complexity of $\mathcal O(N^3)$. 
	This can be considerably improved by employing fast Fourier techniques, as we will see in the next section.

	\section{Non-uniform discrete Fourier transforms}\label{sec:NFFT-Sinkhorn}
	
	The main computational cost of the Sinkhorn algorithm comes from the matrix--vector product with the kernel matrix \eqref{eq:K-matrix}.
	Let $k,\ell\in[K]$ and $\alpha\in\R^{n_{\ell}}$.
	We briefly describe a fast summation method for the computation of 
	$\beta = \Ker^{(k,\ell)}\alpha$, i.e., 
	\begin{equation} \label{eq:fastsum-sum}
		\beta_{i_k}
		=
		\sum_{i_{\ell}=1}^{n_{\ell}}
		\alpha_{i_{\ell}}\,
		\exp \left({-\tfrac1\eta \norm{x^k_{i_k} - x^{\ell}_{i_{\ell}}}_2^2}\right)
		,\quad i_k\in [n_k].
	\end{equation}
	We refer to \cite{postni04} for a detailed derivation and error estimates.
	The main idea is to approximate the kernel function
	\begin{equation} \label{eq:kappa}
		\kappa(x)
		\coloneqq
		\exp ({-\tfrac1\eta x^2})
		,\quad x\in\R,
	\end{equation}
	by a Fourier series.
	In order to ensure fast convergence of the Fourier series, 
	we extend $\kappa$ to a periodic function of certain smoothness  $p\in\N$.
	Let $\varepsilon_{\mathrm B}>0$ and
	$
	\tau>\varepsilon_{\mathrm B} + \max\left\{\|x^k_{i_k}-x^{\ell}_{i_{\ell}}\|_2^2: i_k\in [n_k],i_{\ell}\in [n_{\ell}]\right\}
	$.
	For $x\in\R$, we define the regularized kernel
	\begin{equation} \label{eq:Ker-reg}
		\kappa_{\mathrm R}(x)
		\coloneqq
		\begin{cases}
			\kappa(x), & \abs{x}\le \tau-\varepsilon_{\mathrm B},\\
			\kappa_{\mathrm B}(x), & \tau-\varepsilon_{\mathrm B}< \abs{x} \le \tau,\\
			\kappa_{\mathrm B}(\tau), & \abs x > \tau,
		\end{cases}
	\end{equation}
	where $\kappa_{\mathrm B}$ is a polynomial of degree $p$ that fulfills the Herminte interpolation conditions
	$
	\kappa_{\mathrm B}^{(r)}(\tau-\varepsilon_{\mathrm B}) 
	= 
	\kappa^{(r)}(\tau-\varepsilon_{\mathrm B})
	$ for 
	$
	r=0,\dots,p-1,
	$
	and
	$
	\kappa_{\mathrm B}^{(r)}(\tau) = 0
	$
	for
	$
	r=1,\dots,p-1,
	$
	see Figure~\ref{fig:reg_kernel}.
	\begin{figure}[ht]
		\centering
		\begin{tikzpicture}
			\begin{axis}[xlabel=$x$, ylabel={}, width=.5\textwidth, height=4cm, xmin=-1, xmax=3, ymin=0, ymax=1, axis x line*=bottom, axis y line*=left, legend pos=outer north east, legend style={cells={anchor=west}}]
				\addplot+[thick,mark=none,blue,domain=-.8:.8] {exp(-4 * x^2};
				\addplot+[thick,mark=none,red,domain=.8:1.2] {1.26471+x*(-2.47375+x*1.23688)};
				\addplot+[thick,mark=none,red,domain=-1.2:-.8] {1.26471-x*(-2.47375-x*1.23688)};
				\addplot+[thick,mark=none,blue,domain=1.2:2.8] {exp(-4 * (x-2)^2};
				\addplot+[thick,mark=none,red,domain=2.8:3.2] {1.26471+(x-2)*(-2.47375+(x-2)*1.23688)};
				\legend{$\kappa(x)$,$\kappa_{\mathrm B}(x)$}
			\end{axis}
		\end{tikzpicture}
		\caption{Regularized kernel $\kappa_{\mathrm R}$ for $\eta=1/4$, periodicity length $\tau=1$, boundary interval $\varepsilon_{\mathrm B}=0.2$ and smoothness $p=1$.
		}
		\label{fig:reg_kernel}
	\end{figure}
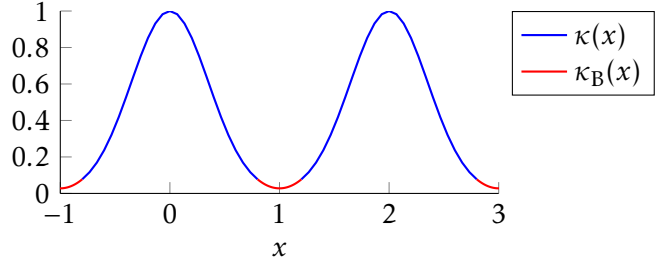
	
	Then we define a $2\tau$-periodic function on $\R^d$ by
	\begin{equation} \label{eq:kappa_tilde}
		\tilde\kappa (x)
		\coloneqq
		\kappa_{\mathrm{R}} (\norm x)
		,\quad x\in [-\tau,\tau)^d.
	\end{equation}
	By construction, $\tilde\kappa$ is $p$ times continuously differentiable
	and we have 
	$\tilde\kappa(x^k_{i_k}-x^{\ell}_{i_{\ell}})= \exp(-{\frac1\eta \norm{\smash{x^k_{i_k}-x^{\ell}_{i_{\ell}}}}_2^2})$ 
	for all 
	$i_k\in[n_k],\,i_{\ell}\in[n_{\ell}]$.
	
	Let $M\in\N$.
	We approximate $\tilde\kappa$ by the $2\tau$-periodic Fourier expansion of degree $2M$, 
	\begin{equation}\label{eq:fourier_serie}
		\tilde\kappa(x)
		\approx
		\sum_{m\in \{-M,\dots,M-1\}^d }
		\hat\kappa(m)\, \exp ({\im \tfrac{\pi}{\tau} m^\tT x}),
		\quad x\in\R^d,
	\end{equation}
	with {the discrete} Fourier coefficients $\hat\kappa(m)\in\C${,
		which can be efficiently approximated by the fast Fourier transform (FFT)
		\begin{equation} \label{eq:fft}
			\hat\kappa(m)
			\coloneqq
			\frac{1}{(2M)^d}
			\sum_{x\in \frac{\tau}{M}\{-M,\dots,M-1\}^d}
			\tilde\kappa(x)\, \exp(-\im \tfrac{\pi}{\tau} m^\top x)
			,\quad m\in \{-M,\dots,M-1\}^d.
		\end{equation}
	}
	
	This yields an approximation of \eqref{eq:fastsum-sum} by
	\begin{multline}
		\beta_{i_k}
		=
		\sum_{j=1}^{n_{\ell}}
		\tilde\kappa
		(x^k_{i_k}-x^{\ell}_{i_{\ell}})
		\, \alpha_{i_{\ell}}
		\approx
		\sum_{i_{\ell}=1}^{n_{\ell}}
		\sum_{m\in\{-M,\dots,M-1\}^d }
		\hat\kappa(m)\,
		\exp \left({\im \tfrac{\pi}{\tau} m^\tT (x^k_{i_k}-x^{\ell}_{i_{\ell}})}\right)
		\, \alpha_{i_{\ell}}
		\\
		=
		\sum_{m\in\{-M,\dots,M-1\}^d }
		\hat\kappa(m)
		\left(
		\sum_{i_{\ell}=1}^{n_{\ell}}
		\alpha_{i_{\ell}}\,
		\exp \left({-\im \tfrac{\pi}{\tau} m^\tT x^{\ell}_{i_{\ell}}}\right)
		\right) 
		\,\exp \left({\im \tfrac{\pi}{\tau} m^\tT x^k_{i_k}}\right).
		\label{eq:fastsum1}
	\end{multline}
	
	The \emph{non-uniform discrete Fourier transform} (NDFT) of $\boldsymbol{\hat\kappa}\coloneqq[\hat\kappa(m)]_{m\in\{-M,\dots,M-1\}^d}$ at the nodes $\Omega^k\subset\R^d$ is defined by
	\begin{equation} \label{eq:ndft}
		[\bF_k \boldsymbol{\hat\kappa}]_{i_{\ell}}
		\coloneqq
		\sum_{m\in\{-M,\dots,M-1\}^d }
		\hat\kappa(m)\,
		\exp \left({-\im \tfrac{\pi}{\tau} m^\tT x^k_{i_{\ell}}}\right)
		,\quad i_{\ell}\in [n_{\ell}],
	\end{equation} 
	and the adjoint NDFT of $\alpha\in\R^{n_{\ell}}$ on the set $\Omega^{\ell}$ is given by
	\begin{equation} \label{eq:ndft*}
		[\bF^*_{\ell} \alpha]_m
		\coloneqq
		\sum_{i_{\ell}=1}^{n_{\ell}}
		\alpha_{i_{\ell}}\,
		\exp \left({-\im \tfrac{\pi}{\tau} m^\tT x^{\ell}_{i_{\ell}}}\right)
		,\quad m\in\{-M,\dots,M-1\}^d,
	\end{equation}
	cf.\ \cite[Sect.\ 7]{PlPoStTa18}.
	Therefore, the approximation \eqref{eq:fastsum1} can be written as
	\begin{equation} \label{eq:fastsum_mat}
		\beta
		=
		\Ker^{(k,\ell)}\alpha
		\approx
		\bF_k (\boldsymbol{\hat\kappa} \odot \bF^*_{\ell} \alpha).
	\end{equation}
	{The procedure is summarized in Algorithm~\ref{alg:fastsum}.}
	
	\begin{algorithm}[ht] 
		\caption{NFFT-based fast summation}\label{alg:fastsum}
		\algblockx[Name]{Precompute}{End}{\textbf{precomputation}}{\textbf{end}}
		\begin{algorithmic}[1]
			\State \textbf{Input:} Vector $\alpha \in \R^{n_\ell}$, kernel function $\kappa$ in \eqref{eq:kappa}, parameters $\varepsilon_{\mathrm B}>0$, $M\in\N$
			\Precompute
			\State Compute the regularized kernel $\kappa_{\mathrm R}$ by \eqref{eq:Ker-reg}
			\State Compute the periodized kernel $\tilde\kappa$ by \eqref{eq:kappa_tilde}
			\State Compute the discrete Fourier coefficients $\hat \kappa(m)$, $m\in\{-M,\dots,M-1\}^d$, by an FFT, see~\eqref{eq:fft}
			\End
			\State Compute the adjoint NDFT $\boldsymbol{F}_\ell \alpha$, see \eqref{eq:ndft*}
			\State Compute the pointwise product $\hat\beta \coloneqq \boldsymbol{\hat\kappa} \odot \boldsymbol{F}_\ell \alpha$
			\State Compute the NDFT $\beta \coloneqq \boldsymbol{F}_k \hat\beta$, see \eqref{eq:ndft}
			\State \Return $\beta \approx \Ker^{(k,\ell)} \alpha$
		\end{algorithmic}
	\end{algorithm}
	
	There are fast algorithms,
	known as \emph{non-uniform fast Fourier transform} (NFFT),
	allowing the computation of an NDFT \eqref{eq:ndft} and its adjoint \eqref{eq:ndft*} in $\mathcal O(M^d\log M + N)$ steps
	up to arbitrary numeric precision,
	see, e.g., \cite{bey95,duro95} and \cite[Sect. 7]{PlPoStTa18},
	where $N=\norm{\bn}_\infty$.
	Note that the direct implementation of \eqref{eq:fastsum1} requires $\mathcal O(M^d N)$ operations.
	We call the Sinkhorn algorithm where the matrix--vector multiplication is performed via \eqref{eq:fastsum_mat} the \emph{NFFT-Sinkhorn algorithm}.
	Provided we fix the Fourier expansion degree $M$, which is possible because of the smoothness of $\kappa$,
	we end up at a numerical complexity of $\mathcal O(KN)$ for each iteration step of the NFFT-Sinkhorn algorithm for trees.
	In case of a circle (Algorithm \ref{alg:sink-circle}), we can apply the fast summation column by column for the matrix--matrix product with $\Ker^{(k,k+1)}$, yielding a complexity of $\mathcal O(KN^2)$ for each iteration step.
	
	\section{Numerical examples}\label{sec:Numerical_examples}
	
	We illustrate the results from section~\ref{sec:NFFT-Sinkhorn} concerning the Sinkhorn algorithm and its accelerated version, the NFFT-Sinkhorn.
	First, we investigate the effect of parameter choices in some artificial examples.
	Then, we look at the one-dimensional Euler flow problem of incompressible fluids and the fixed support barycenter problem of images.\footnote{The code for our examples is available at \href{https://github.com/fatima0111/NFFT-Sinkhorn}{https://github.com/fatima0111/NFFT-Sinkhorn}.}
	All computations were performed on an 8-core Intel Core i7-10700 CPU and 32GB memory.
	For computing the NFFT of section~\ref{sec:NFFT-Sinkhorn}, we rely on the implementation \cite{KeKuPo09}.
	
	\subsection{Uniformly distributed points}
	We consider the $\MOTe$ problem for uniform measures $\mu^k$ of uniformly distributed points on $\Omega=[-1/2,1/2]$.
	We chose the entropy regularization parameter $\eta=0.1$, so that a boundary regularization \eqref{eq:Ker-reg} for the fast summation method is necessary. 
	For the tree-structured cost function, we set the boundary regularization $\varepsilon_{\mathrm{B}}=1/16$, the Fourier expansion degree $M=156,$ and the smoothness parameters $p=3$, see section~\ref{sec:NFFT-Sinkhorn}.
	In Figure~\ref{fig:tree_structure_time_complexity} left,
	we see the linear dependence of the computational time on the number $K$ of marginals.
	For a growing number $N$ of points, 
	the NFFT-Sinkhorn algorithm, which requires $\mathcal O(KN)$ steps, clearly outperforms the standard method, which requires $\mathcal O(KN^2)$ steps,
	see figure~\ref{fig:tree_structure_time_complexity} right.
	
	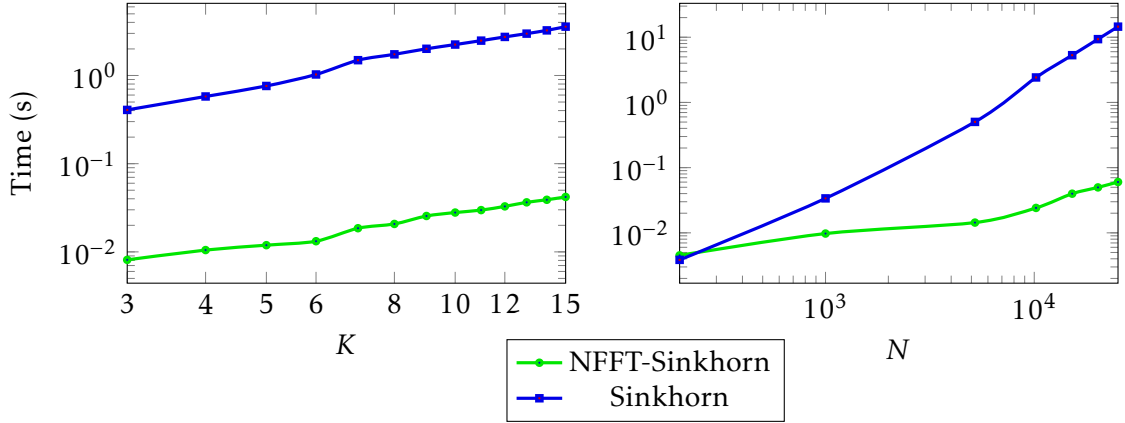
\begin{figure}[ht]
		\centering
		\begin{tikzpicture}[smooth] 
			\begin{groupplot}[group style={
					group name=my plots,
					group size=2 by 1,
					xlabels at=edge bottom,
					horizontal sep=1.5cm,vertical sep=3cm,},		
				legend style={at={(1.5,-.2)},anchor=north east},
				subtitle/.style={title=\gpsubtitle{#1}}]
				\nextgroupplot[xlabel=$K$, ylabel=Time (s), width={0.5\linewidth}, height={0.36\linewidth}, 
				ymode=log, xmode=log, xmin=3, xmax = 15, 
				xtick={3,4,5,6,8,10,12,15},
				xticklabels={3,4,5,6,8,10,12,15},
				]
				\addlegendentry{NFFT-Sinkhorn}
				\addplot+[very thick,green!90!black, mark size=1pt,
				] table[x index=0,y index=1]{data/complexity_tree_with_n_p_10000.dat};
				\addlegendentry{Sinkhorn}
				\addplot+[very thick,blue!90!black, mark size=1pt,
				] table[x index=0,y index=2]{data/complexity_tree_with_n_p_10000.dat};
				
				\nextgroupplot[xlabel=$N$,  width={0.5\linewidth}, height={0.36\linewidth},  ymode=log, xmode=log, xmin=200, xmax=25200, 
				]
				\addplot+[very thick,green!90!black, mark size=1pt,
				] table[x index=0,y index=1]{data/complexity_tree_with_K_10_eta_0.1.dat};
				\addplot+[very thick,blue!90!black, mark size=1pt,
				] table[x index=0,y index=2]{data/complexity_tree_with_K_10_eta_0.1.dat};
			\end{groupplot}
		\end{tikzpicture}
		\caption{Computation time { in seconds} of $\MOTe$ with tree-structured cost function with regularization parameter $\eta = 0.1.$ Left: fixed  $N=10^4.$ Right: fixed $K=10.$  }
		\label{fig:tree_structure_time_complexity}
	\end{figure}
	
	In Figure \ref{fig:circle_structure_time_complexity}, we show the computation times for the circle-structured cost function, with the parameters $\eta = 0.1$, $\varepsilon_{\mathrm{B}}=3/32$, $M=2000$, and $p=3$.
	As the Sinkhorn iteration requires matrix--matrix products,
	it is more costly than for the tree-structured cost function
	and so we used a lower number of points $N$.
	The advantage of the NFFT-Sinkhorn is smaller than for the tree, but still considerable.
	We point out here that the fast summation method is applied column by column to the matrix--matrix product,
	and the Fourier expansion degree $M$ is larger.	
	
	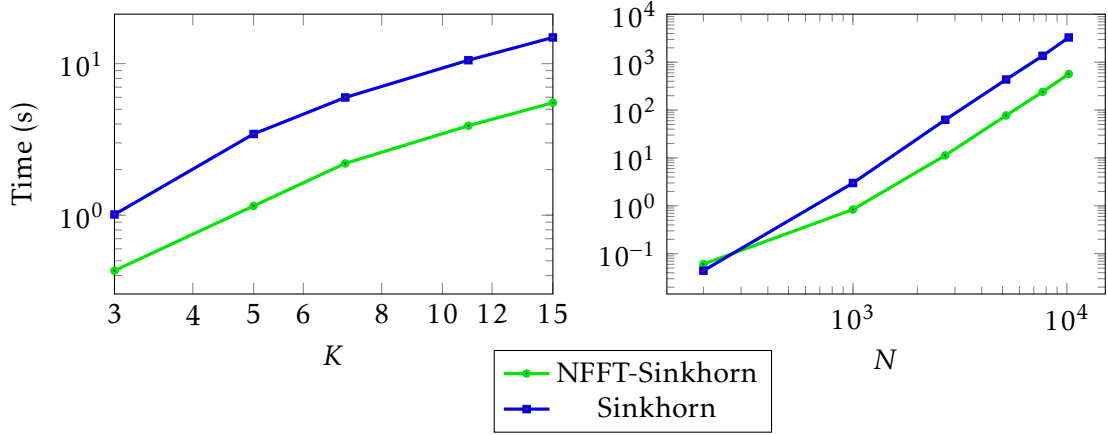
\begin{figure}[htb]
		\centering
		\begin{tikzpicture}
			\begin{groupplot}[group style={
					group name=my plots,
					group size=2 by 1,
					xlabels at=edge bottom,
					horizontal sep=1.5cm,vertical sep=3cm},
				legend style={at={(1.5,-.2)},anchor=north east},
				subtitle/.style={title=\gpsubtitle{#1}}]
				\nextgroupplot[xlabel=$K$, ylabel=Time (s), width={0.5\linewidth}, height={0.36\linewidth}, 
				ymode=log, xmode=log, xmin=3, xmax = 15,  
				xtick={3,4,5,6,8,10,12,15},
				xticklabels={3,4,5,6,8,10,12,15},
				]
				\addlegendentry{NFFT-Sinkhorn}
				\addplot+[very thick,green!90!black, mark size=1pt,
				] table[x index=0,y index=1]{data/circle_complexity_wrt_K_15_with_n_p_700.dat};
				\addlegendentry{Sinkhorn}
				\addplot+[very thick,blue!90!black, mark size=1pt,
				] table[x index=0,y index=2]{data/circle_complexity_wrt_K_15_with_n_p_700.dat};
				\nextgroupplot[xlabel=$N$,  width={0.5\linewidth}, height={0.36\linewidth},  ymode=log, xmode=log, 
				]
				\addplot+[very thick,green!90!black, mark size=1pt,
				] table[x index=0,y index=1]{data/circle_complexity_wrt_n_p_10200_with_K_3.dat};
				\addplot+[very thick,blue!90!black, mark size=1pt,
				] table[x index=0,y index=2]{data/circle_complexity_wrt_n_p_10200_with_K_3.dat};
			\end{groupplot}
		\end{tikzpicture}
		\caption{Computation time { in seconds} of $\MOTe$ with circle-structured cost function with regularization parameter $\eta = 0.1.$ Left: fixed  $N=700.$ Right: fixed $K=3.$}
		\label{fig:circle_structure_time_complexity}
	\end{figure}

	{Finally, we investigate how the approximation error between the Sinkhorn and NFFT-Sinkhorn algorithm, $|{\mathcal{S}^{(r)}-\tilde{S}^{(r)}}|$, depends on the entropy regularization parameter $\eta$ and the Fourier expansion degree $M$ at a fixed iteration $r=10$, where $\mathcal{S}^{(r)}$ denotes the evaluation with the Sinkhorn algorithm and $\tilde{\mathcal{S}}^{(r)}$ its evaluation with the NFFT-Sinkhorn algorithm. 
		Since the time differences for different $M$ in the 1-dimensional case are very small,
		we consider $2$-dimensional uniform marginal measures for the $\MOTe$ problems with tree- or circle-structured cost functions.
		In Figures~\ref{fig:tree_error_time_wrt_M} and \ref{fig:circle_error_time_wrt_M},
		we see that for smaller $\eta$, we need a larger expansion degree $M$ to achieve a good accuracy.
		The error stagnates at a certain level and does not decrease anymore for increasing $M$.
		This could be improved by increasing the approximation parameter $p$ and the cutoff parameter of the NFFT, cf.\ \cite[Section 7]{PlPoStTa18}.
		Parameter choice methods for the NFFT-based summation were discussed in \cite{Ne15}.
		For an appropriately chosen $M$, the NFFT-Sinkhorn is usually much faster than the Sinkhorn algorithm.
		However, for very small $\eta$, the kernel function \eqref{eq:kappa} is concentrated on a small interval and therefore a simple truncation of the sum \eqref{eq:fastsum-sum} might be beneficial to the NFFT approximation.
	}
	\begin{figure}[htb]
		\centering
		\begin{tikzpicture}
			\begin{groupplot}[group style={
					group name=my plots,
					group size=2 by 1,
					xlabels at=edge bottom,
					ylabels at=edge left,
					horizontal sep=.5cm,vertical sep=3cm,}]
				\nextgroupplot[ width={0.5\linewidth}, height={0.36\linewidth}, ymode=log, xmode=log, xmin=8, xmax = 512,  
				xtick={8, 16, 32, 64, 128, 256, 512},
				xticklabels={8, 16, 32, 64, 128, 256, 512}, xlabel={ $M$}, ylabel={Error $|{\mathcal{S}^{(10)}-\tilde{S}^{(10)}}|$},legend style={at={(.96,-.3)},anchor=north east, legend columns=2}]
				\addlegendentry{$\eta=0.005$}
				\addplot+[very thick,orange!90!black, mark size=1pt, 
				] table[x index=0,y index=3]{data/tree_error_time_tradeoff_wrt_M_K_10_with_n_p_10000.dat};
				\addlegendentry{$\eta=0.05$}
				\addplot+[very thick, mark size=.5pt, red, 
				] table[x index=0,y index=5]{data/tree_error_time_tradeoff_wrt_M_K_10_with_n_p_10000.dat};
				\addlegendentry{$\eta=0.5$}
				\addplot+[very thick, mark size=.5pt, black, 
				] table[x index=0,y index=7]{data/tree_error_time_tradeoff_wrt_M_K_10_with_n_p_10000.dat};
				\nextgroupplot[xlabel=$M$, ylabel=Time (s), yticklabel pos=right,  width={0.5\linewidth}, height={0.36\linewidth},  ymode=log, xmode=log, 
				xmin=8, xmax = 512,
				xtick={8, 16, 32, 64, 128, 256, 512},
				xticklabels={8, 16, 32, 64, 128, 256, 512},legend 	style={at={(.8,-.3)},anchor=north east,legend columns=1}
				]
				\addlegendentry{NFFT-Sinkhorn}
				\addplot+[very thick,green!90!black, mark size=1pt,
				] table[x index=0,y index=1]{data/tree_error_time_tradeoff_wrt_M_K_10_with_n_p_10000.dat};
				\addlegendentry{Sinkhorn}
				\addplot+[very thick,blue!90!black, mark size=1pt,
				] table[x index=0,y index=2]{data/tree_error_time_tradeoff_wrt_M_K_10_with_n_p_10000.dat};
			\end{groupplot}
		\end{tikzpicture}
		\caption{{$\MOTe$ with tree-structured cost function, where $d=2$, $K=10,$ $N=10000,$ and $p=3.$ Left: Approximation error of $\mathcal S^{(10)}$ between Sinkhorn and NFFT-Sinkhorn algrithm depending on the number of Fourier coefficients $M$ of the NFFT. Right: Computation time in seconds.}}
		\label{fig:tree_error_time_wrt_M}
	\end{figure}
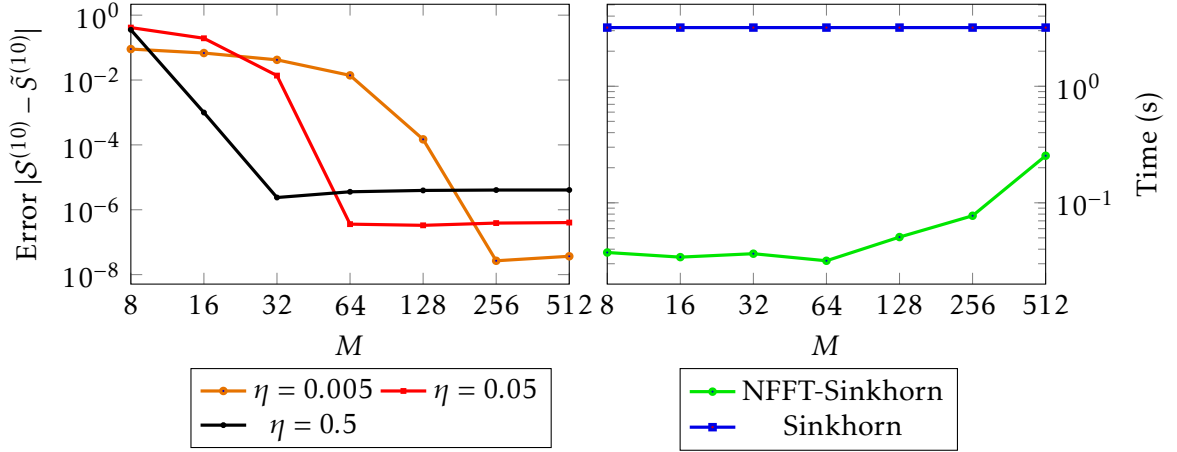
	
	\begin{figure}[htb]
		\centering
		\hspace{-5mm}
		\begin{tikzpicture}
			\begin{groupplot}[group style={
					group name=my plots,
					group size=2 by 1,
					xlabels at=edge bottom,
					horizontal sep=.5cm,vertical sep=3cm},
				subtitle/.style={title=\gpsubtitle{#1}}]
				\nextgroupplot[ width={0.5\linewidth}, height={0.36\linewidth}, ymode=log, xmode=log, xmin=8, xmax = 512,  
				xtick={8, 16, 32, 64, 128, 256, 512},
				xticklabels={8, 16, 32, 64, 128, 256, 512}, xlabel={ $M$}, ylabel={Error $|{\mathcal{S}^{(10)}-\tilde{S}^{(10)}}|$},legend style={at={(.96,-.3)},anchor=north east, legend columns=2}]
				\addlegendentry{$\eta=0.005$}
				\addplot+[very thick,orange!90!black, mark size=1pt, 
				] table[x index=0,y index=3]{data/circle_error_Time_tradeoff_wrt_M_K_3_with_n_p_1000.dat};
				\addlegendentry{$\eta=0.05$} 
				\addplot+[very thick, mark size=1pt, red, 
				] table[x index=0,y index=5]{data/circle_error_Time_tradeoff_wrt_M_K_3_with_n_p_1000.dat};
				\addlegendentry{ $\eta=0.5$}
				\addplot+[very thick, mark size=1pt, black, 
				] table[x index=0,y index=7]{data/circle_error_Time_tradeoff_wrt_M_K_3_with_n_p_1000.dat};
				\nextgroupplot[xlabel=$M$, ylabel=Time (s), yticklabel pos=right,  width={0.5\linewidth}, height={0.36\linewidth},  ymode=log, xmode=log, 
				xmin=8, xmax = 512,
				xtick={8, 16, 32, 64, 128, 256, 512},
				xticklabels={8, 16, 32, 64, 128, 256, 512},legend style={at={(.8,-.3)},anchor=north east,legend columns=1}
				]
				\addlegendentry{NFFT-Sinkhorn}
				\addplot+[very thick,green!90!black, mark size=1pt,
				] table[x index=0,y index=1]{data/circle_error_Time_tradeoff_wrt_M_K_3_with_n_p_1000.dat};
				\addlegendentry{Sinkhorn}
				\addplot+[very thick,blue!90!black, mark size=1pt,
				] table[x index=0,y index=2]{data/circle_error_Time_tradeoff_wrt_M_K_3_with_n_p_1000.dat};
			\end{groupplot}
		\end{tikzpicture}
		\caption{	{$\MOTe$ with circle-structured cost function, where $d=2$, $K=3,$ $N=1000,$ and $p=3.$ Left: Approximation error of $\mathcal S^{(10)}$ between Sinkhorn and NFFT-Sinkhorn algrithm depending on the number of Fourier coefficients $M$. Right: Computation time in seconds.}}
		\label{fig:circle_error_time_wrt_M}
	\end{figure}
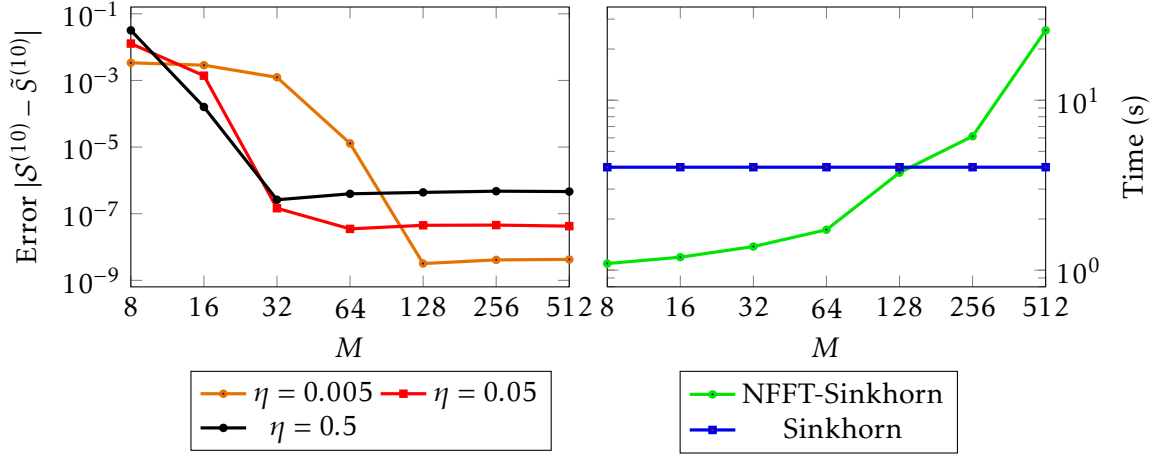
	
	\subsection{Fixed-support Wasserstein barycenter for general trees}
	
	Let $\mathcal{T}=\left(\mathcal{V}, \mathcal{E}\right)$ be a tree with set of leaves	$\mathcal{L}$, see section~\ref{sec:tree}.
	For $k\in\mathcal{L}$, let measures~$\mu^k$
	and weights~$\tilde{w}_k \in [0,1]$ be given that satisfy $\sum_{k \in \mathcal{L}} \tilde{w}_k = 1$.
	For any edge $e\in \mathcal{E}$, we set
	\begin{equation}
		w_e \coloneqq \begin{cases}
			1 & \text{if } \abs{e\cap \mathcal{L}} \neq 1 ,\\
			\tilde{w}_{k}& \text{if } e\cap \mathcal{L}= \set{k}.
		\end{cases}
	\end{equation}
	The generalized barycenters are the minimizers $\mu^k$, $k\in\mathcal V\setminus\mathcal L$, of
	\begin{equation}\label{eq:barycenter_optimal_problem_tree}
		\inf_{{\mu^k\in \mathbb{P}\left(\Omega^k\right), k\in \mathcal{V}\setminus \mathcal{L}}} \sum_{e=\set{e_1, e_2} \in \mathcal{E}} w_e \mathrm{W}_2^2\left(\mu^{e_1}, \mu^{e_2}\right),
	\end{equation}
	where $\mathrm{W}_2^2(\mu^{e_1}, \mu^{e_2})$ is the squared Wasserstein distance \cite{BassettiGualandiVeneroni2018,CuturiDoucet2014} between the measures $\mu^{e_1}$ and $\mu^{e_2}.$ The well-known Wasserstein barycenter problem \cite{GabrielJulieMarcJulien2012, vonLindheim2022} is a special case of \eqref{eq:barycenter_optimal_problem_tree}, where the tree is star-shaped and the barycenter corresponds to the unique internal node. 
	We consider the fixed-support barycenter problem  \cite{BenamouCarlierCuturi2015, TakezawaSatoKozarevaRavi2021},
	where also the nodes $\bx^k$, $k\in\mathcal V\setminus\mathcal L$ are given,
	so that we need to optimize \eqref{eq:barycenter_optimal_problem_tree} only for $\mu^k_{i_k}$, $k\in\mathcal V\setminus\mathcal L$.
	This yields an MOT problem with the tree-structured cost
	\begin{equation}\label{eq:cost_function_barycenter_tree}
		C_{\bi} = \sum_{e=\set{e_1,e_2}\in \mathcal{E}} w_e \norm{\bx^{e_1}_{i_{e_1}}- \bx^{e_2}_{i_{e_1}}}^2_2,
		\quad \bi \in \bI,
	\end{equation}
	where the marginal constraints of \eqref{eq:feasible_discrete_plan} are only set for the known measures $\mu^k$, $k\in \mathcal{L}$, 
	see \cite{AguehCarlier2011}.
	This barycenter problem can be solved approximately using a modification of the Sinkhorn algorithm~\ref{alg:sink-tree}
	in which we replace line 12 by 
	\begin{equation}
		(\phi^k)^{(r+1)} = \begin{cases}
			\mathbf{1} & \text{if } k \in \mathcal{L},\\
			{\mu^k} \oslash \left({\gamma_k^{(r)} \odot \left(\beta^{\odot_{\mathcal{C}_{k}}}\right)^{(r)}}\right)& \text{otherwise.}
		\end{cases}
	\end{equation}
	
	\begin{figure}
		\centering
		\includegraphics[width=0.2\textwidth]{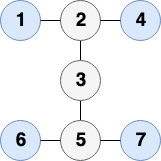}
		\caption{The tree graph of the barycenter problem, with leaves $\mathcal L$ marked in blue.}
		\label{fig:h_shaped}
	\end{figure}
	We test our algorithm with a tree consisting of $K=7$ nodes, see Figure \ref{fig:h_shaped}.
	The four given marginals $\mu^k$, $k\in\mathcal L$, are dithered images in $\R^2$ with uniform weights $\mu^k_{i_k} = 1/N$.
	As support points of the barycenters $\mu^k$, $k\in\mathcal V\setminus\mathcal L$, we take the union over all support points $\bx^k_{i_k}$ of all four input measures $\mu^k,$ $k\in \mathcal{L}.$ 
	Furthermore, we use the barycenter weights $\tilde w^k=1/4$.
	The given images and the computed barycenters are show in Figure~\ref{fig:barycenter_general_tree},
	where we executed $r = 150$ iterations of the Sinkhorn algorithm and its accelerated version.
	We chose the regularization parameter $\eta=5\cdot10^{-3}$ and the fast summation parameters $M=156$, $p=3.$
	
	\begin{figure}
		\centering
		\begin{minipage}[c]{\textwidth}
			\centering
			\subcaptionbox{Sinkhorn}
			{\includegraphics[width=0.46\textwidth]{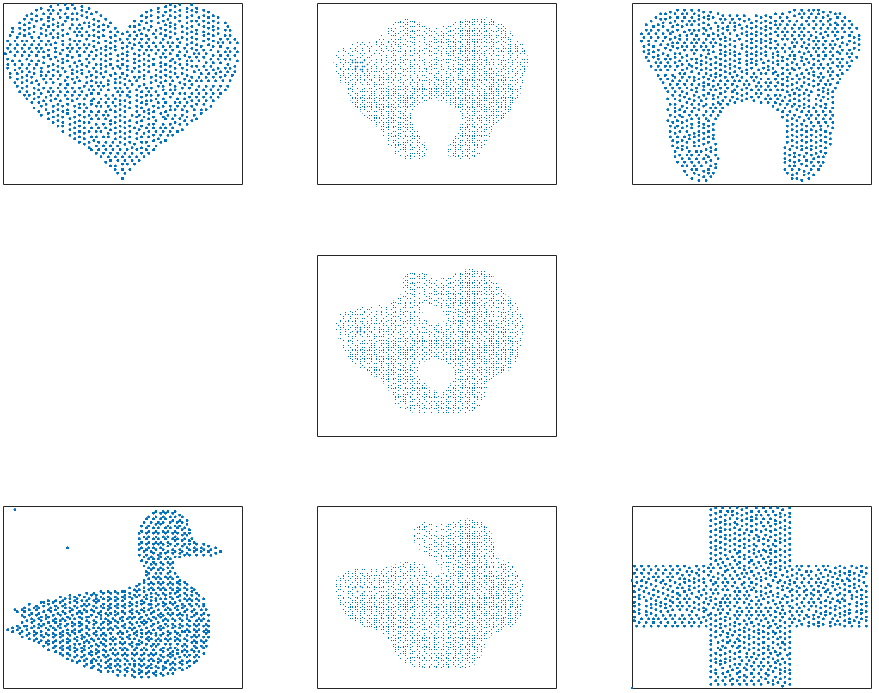}}\hfill
			\subcaptionbox{NFFT-Sinkhorn}
			{\includegraphics[width=0.46\textwidth]{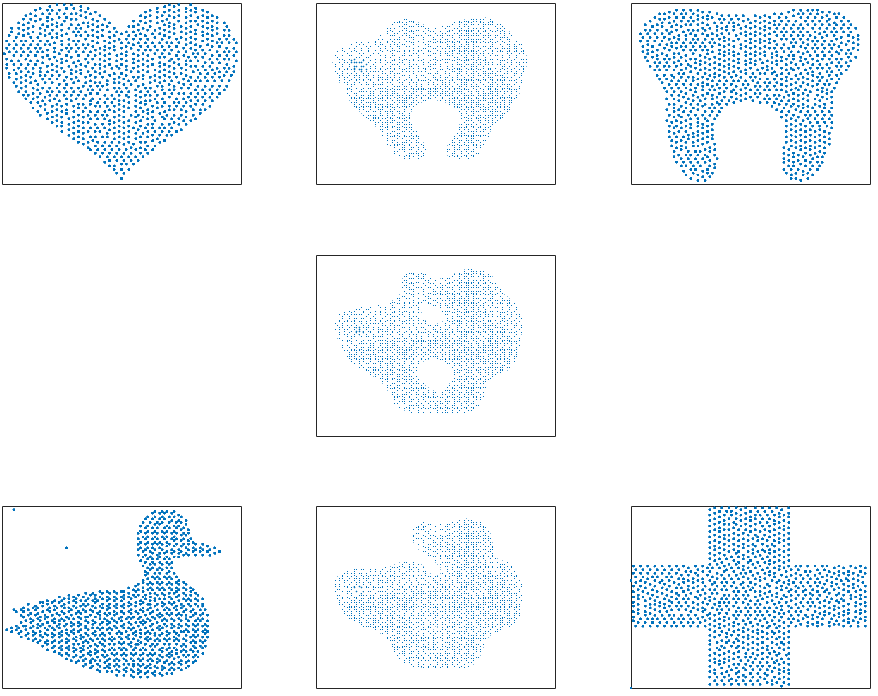}}
		\end{minipage}
		\caption{Given measures: $\set{1,4,6,7}, $ entropy regularization parameter $\eta=5\cdot 10^{-3}$, $r=150$, $\tilde{w}=\frac{1}{4}(1,1,1,1).$ 
			NFFT-Sinkhorn parameters: $M=156$, $p=3.$
			The test images are taken from \cite{flamary2021pot}.}
		\label{fig:barycenter_general_tree}
	\end{figure}

	\subsection{Generalized Euler flows}
	
	We consider the motion of $N$ particles of an incompressible fluid, e.g.\ water, in a bounded domain $\Omega$ 
	in discrete time steps $t_k\coloneqq (k-1)/(K-1)$, $k \in [K]$.
	We assume that we know the function $\sigma\colon\Omega\to\Omega$, which connects $N$ initial positions $x_{i_1}\in\Omega$ of particles with their final positions $x_{i_K}\in\Omega$.
	At each time step $t_k$, we know an image of the particle distribution,
	which is described by the discrete marginal measure~$\mu^k$, $k\in [K]$, with uniform weights $\mu^k_{i_k} = {1}/{N}$.
	We want to find out how the single particles move, i.e., their trajectories.
	Due to the least-action principle, this problem can be formulated as MOT problem \eqref{eq:MOT} 
	with the circle-structured cost  
	\begin{equation*}
		C_{\bi} 
		= \norm{x_{i_K}^K-\sigma\left(x_{k_1}^1\right)}_2^2+ \sum_{k=1}^{K-1} \norm{x_{i_{k+1}}^{k+1}-x_{i_k}^k}_2^2,
		\quad \bi\in\bI,
	\end{equation*}
	see \cite{YannBrenier1989,Brenier1993,Brenier97minimalgeodesics}.
	Then the pair marginal  
	\begin{equation}
		\boldsymbol{\Pi}_{1,k}
		\coloneqq
		\sum_{\ell\in [K]\setminus{\set{1,k}}}\sum_{i_{\ell} \in [n_{\ell}]} \boldsymbol{\Pi}_{i_1,\dots, i_K}
	\end{equation}
	of the optimal plan $\boldsymbol{\Pi}$
	provides the (discrete) probability that a particle which was initially at position $x_{i_1} \in \Omega$ is in position $x_{i_k}\in\Omega$ at time $t_k$, $k=2,\dots,K-1$. 
	The one-dimensional problem has been studied by several authors~\cite{AltschulerBoix-AdseraPolynomial, BenamouCarlierCuturi2015, BenamouCarlierNenna20217}, where the particles are assumed to be on a grid. 
	Here, we consider the case where the positions are uniformly distributed.
	We draw $N=400$ uniformly distributed points on $\Omega=[0,1]$. We use $K=5$ marginal constraints and the entropy regularization parameter $\eta = 0.05$.  Figures~\ref{fig:Euler-flow-Sinkhorn-map1} and \ref{fig:Euler-flow-Sinkhorn-map2} display the probability matrix $\boldsymbol{\Pi}_{1,k}$ describing the motion of the particles from initial time $t_1=0$ to time $t_k$,
	where we use $r = 50$ iterations for both the NFFT-Sinkhorn and the Sinkhorn algorithms, and two different connection functions $\sigma$.
	
	\begin{figure}[!ht]
		\centering
		\begin{minipage}[c]{0.17\linewidth}
			\centering
			{${{t_1=0}}$\par\medskip
				\includegraphics[width=\textwidth]{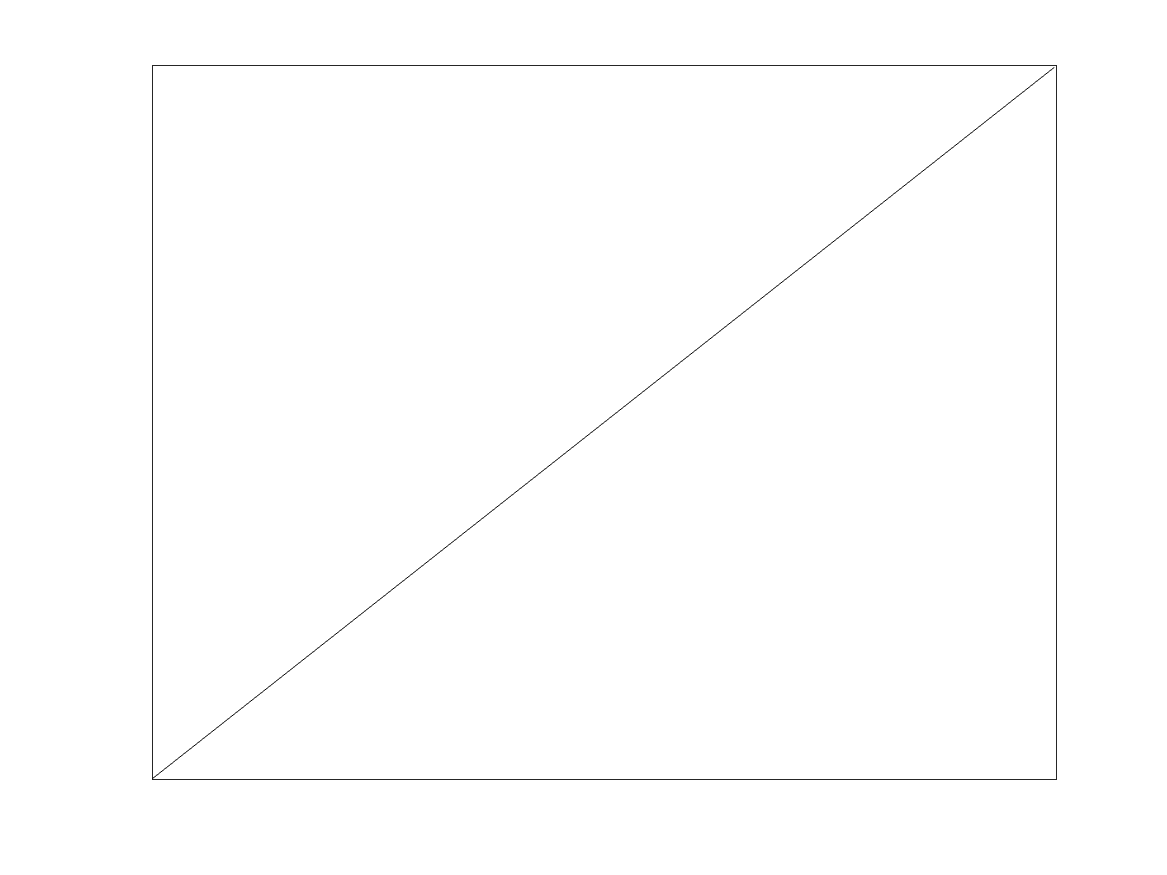}}\quad
		\end{minipage}
		\begin{minipage}[c]{0.15\textwidth}
			\centering
			{${{t_2=1/5}}$\par\medskip
				\includegraphics[width=\textwidth]{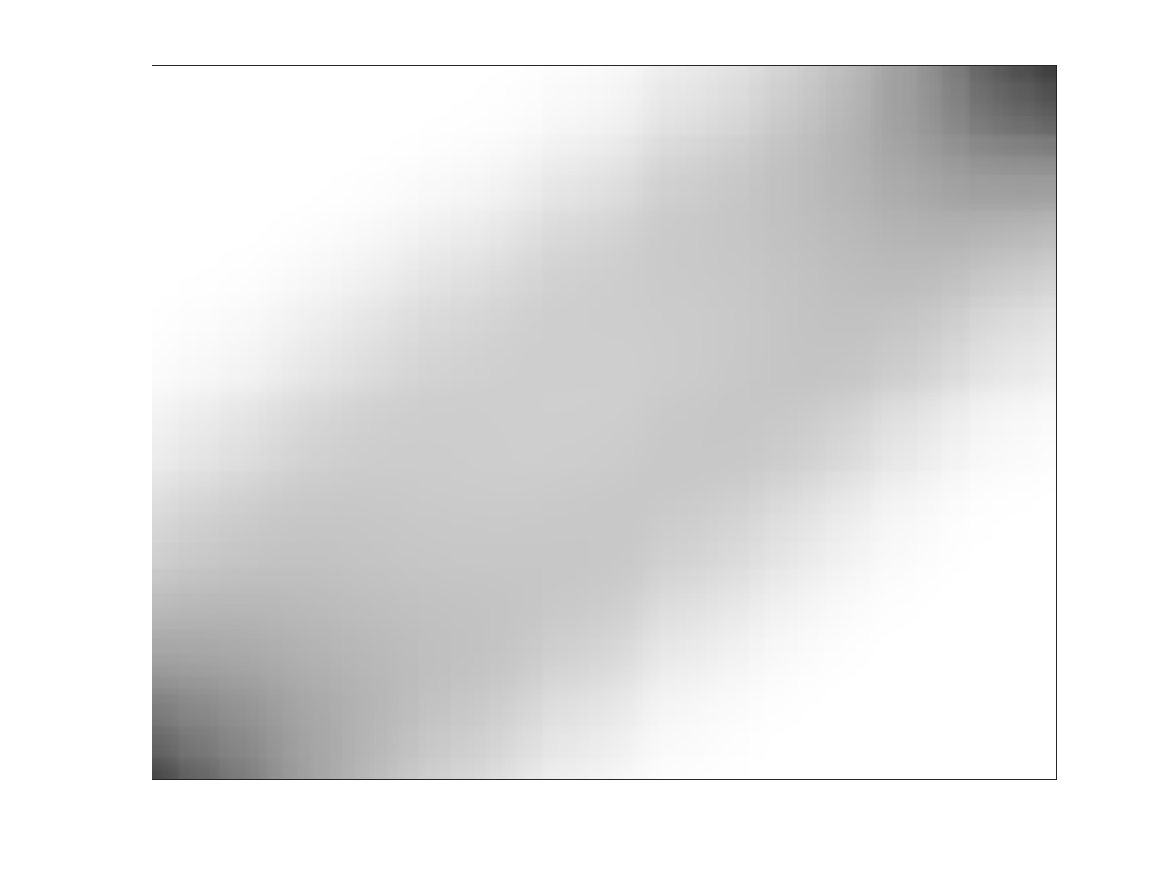}}\quad
			{\includegraphics[width=\textwidth]{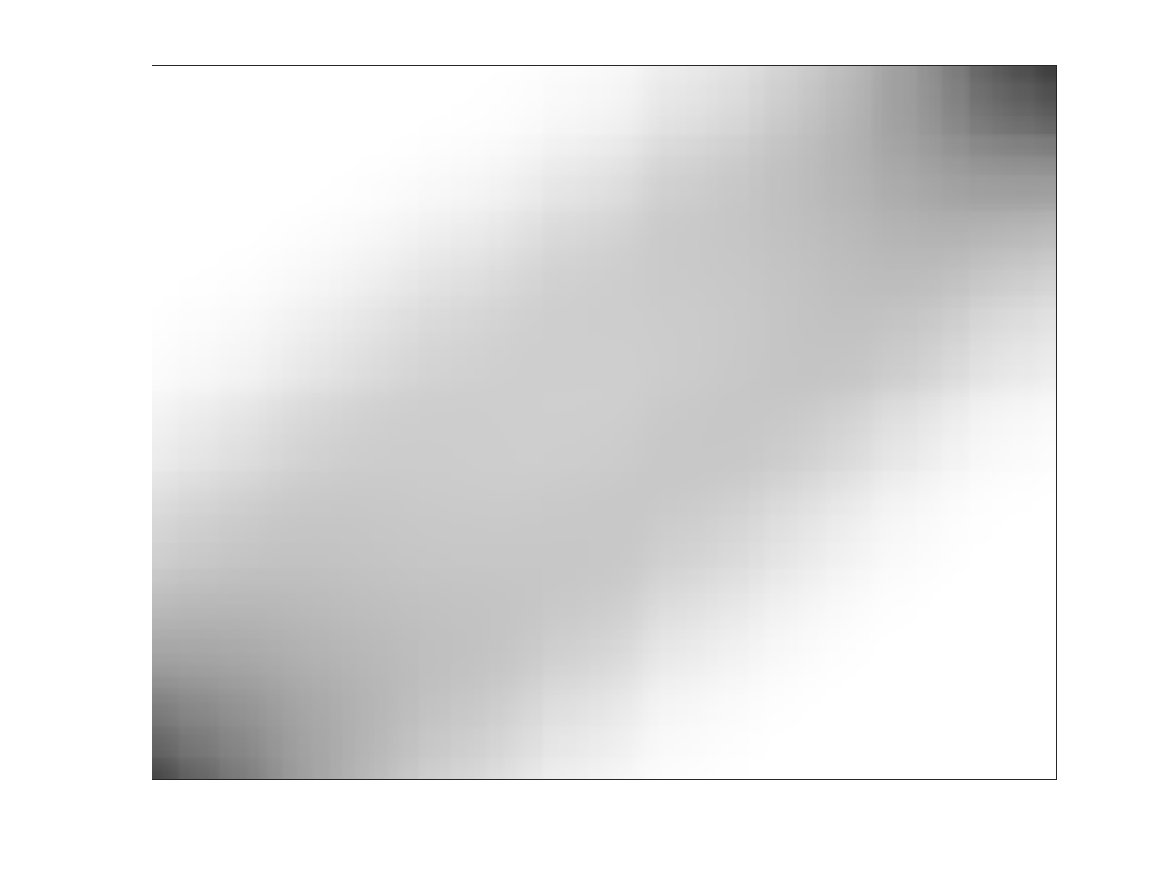}}
		\end{minipage}
		\begin{minipage}[c]{0.15\linewidth}
			\centering
			{${{t_3=2/5}}$\par\medskip
				\includegraphics[width=\textwidth]{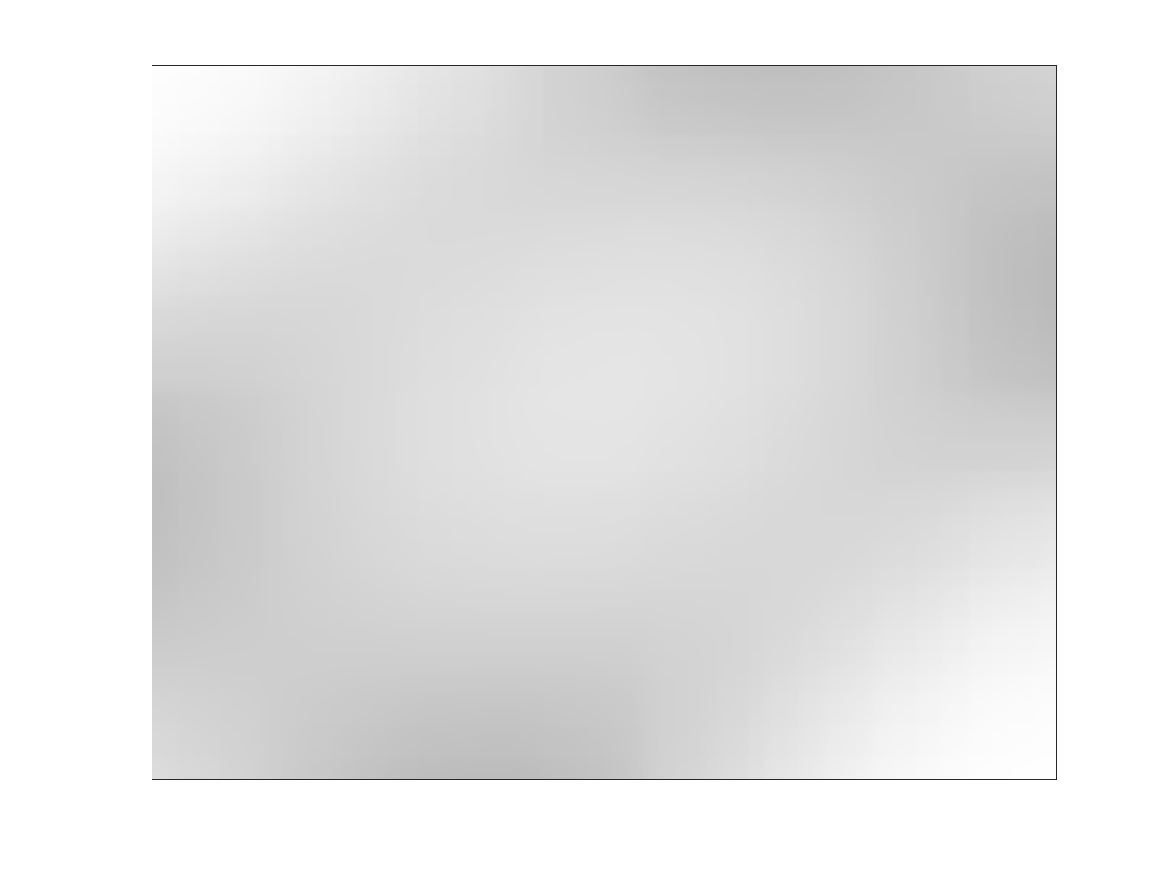}}\quad
			{
				\includegraphics[width=\textwidth]{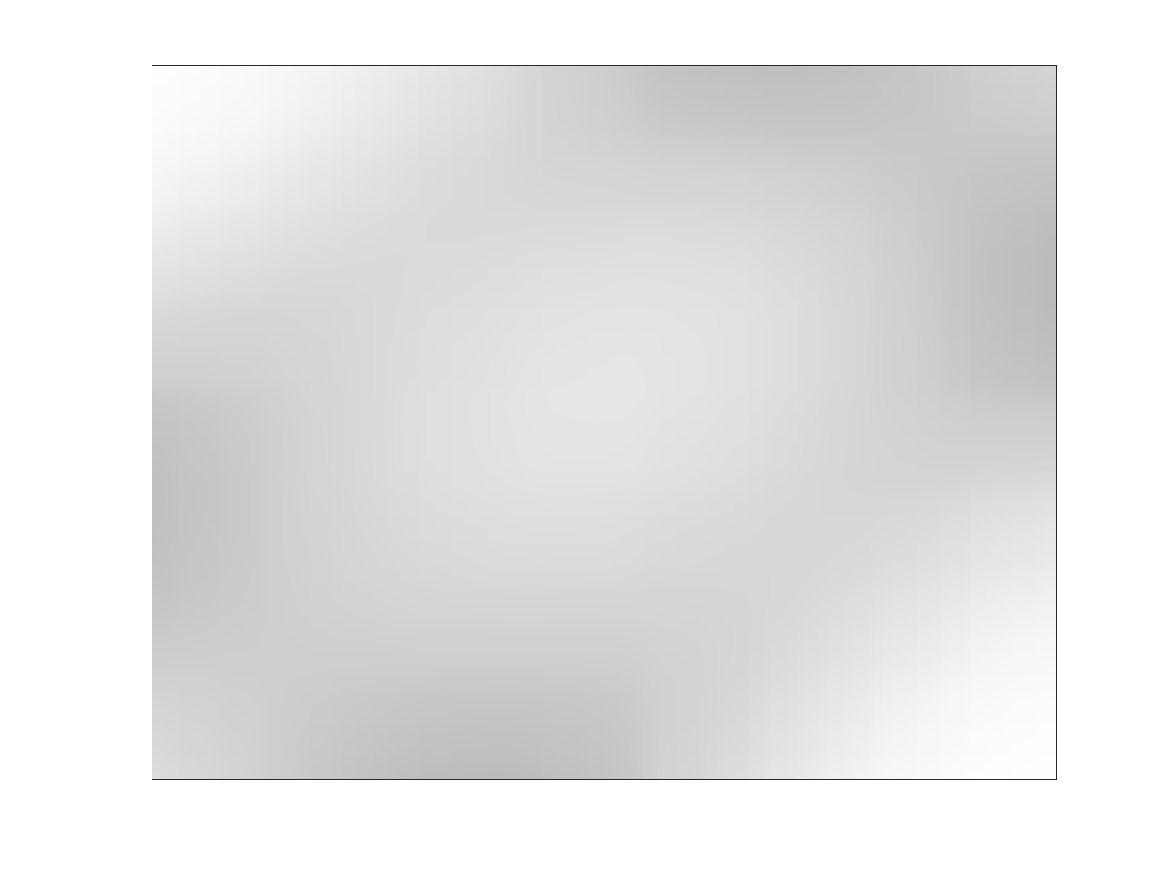}}
		\end{minipage}
		\begin{minipage}[c]{0.15\linewidth}
			\centering
			{${{t_5=3/5}}$\par\medskip
				\includegraphics[width=\textwidth]{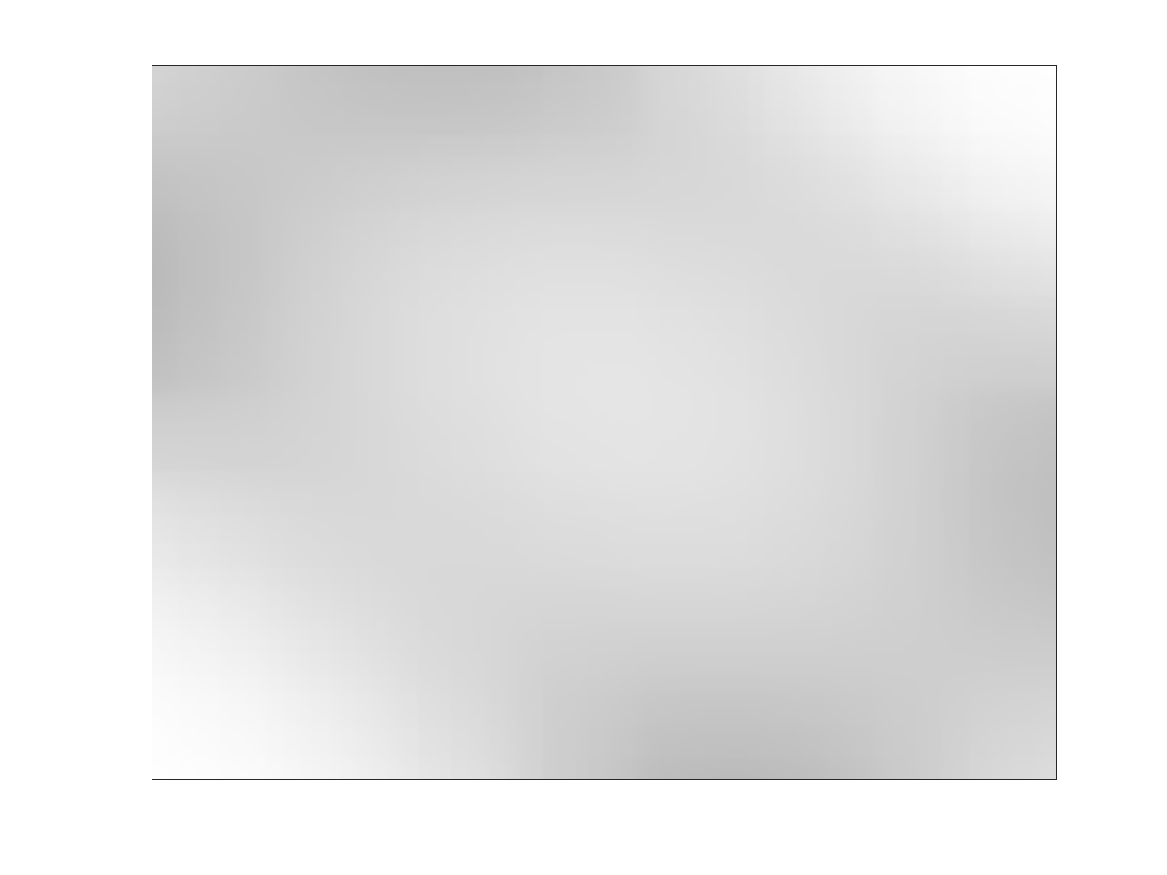}}\quad
			{\includegraphics[width=\textwidth]{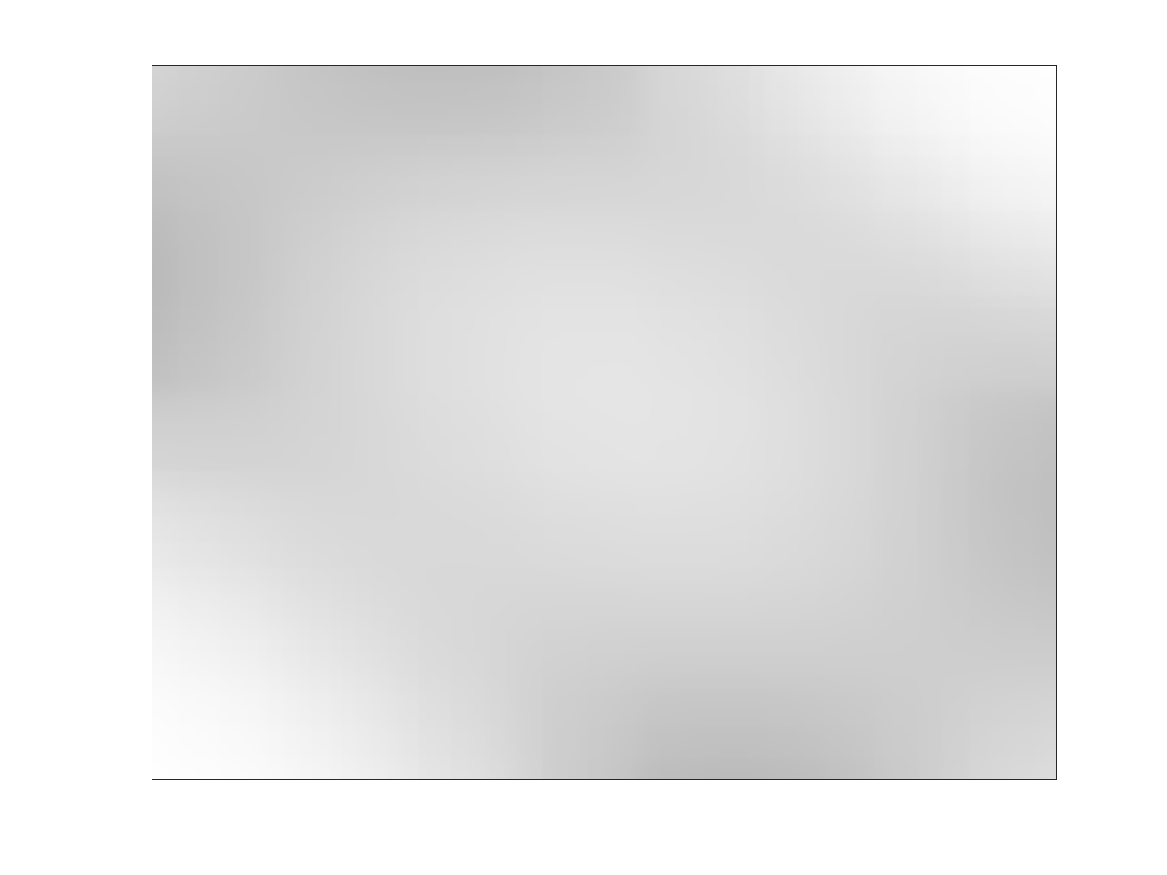}}
		\end{minipage}
		\begin{minipage}[c]{0.15\linewidth}
			\centering
			{${{t_5=4/5}}$\par\medskip
				\includegraphics[width=\textwidth]{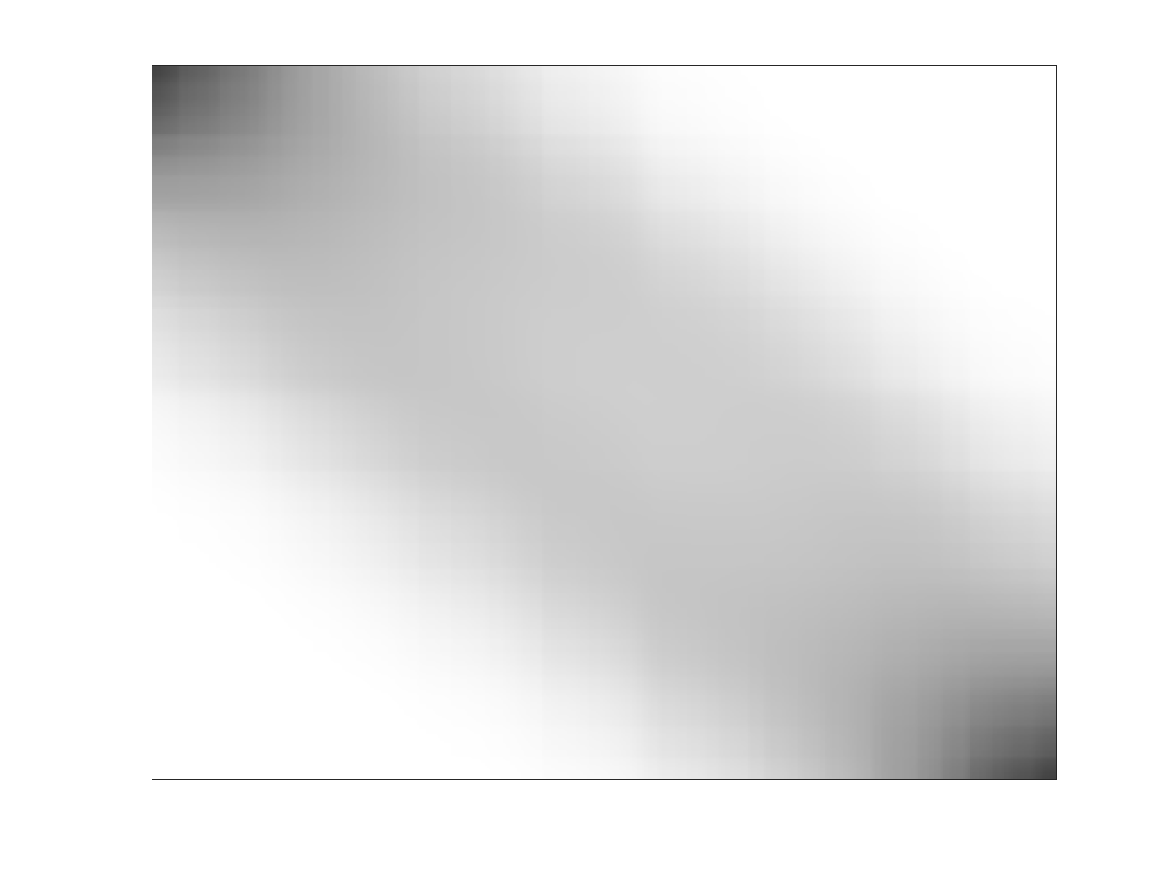}}\quad
			{\includegraphics[width=\textwidth]{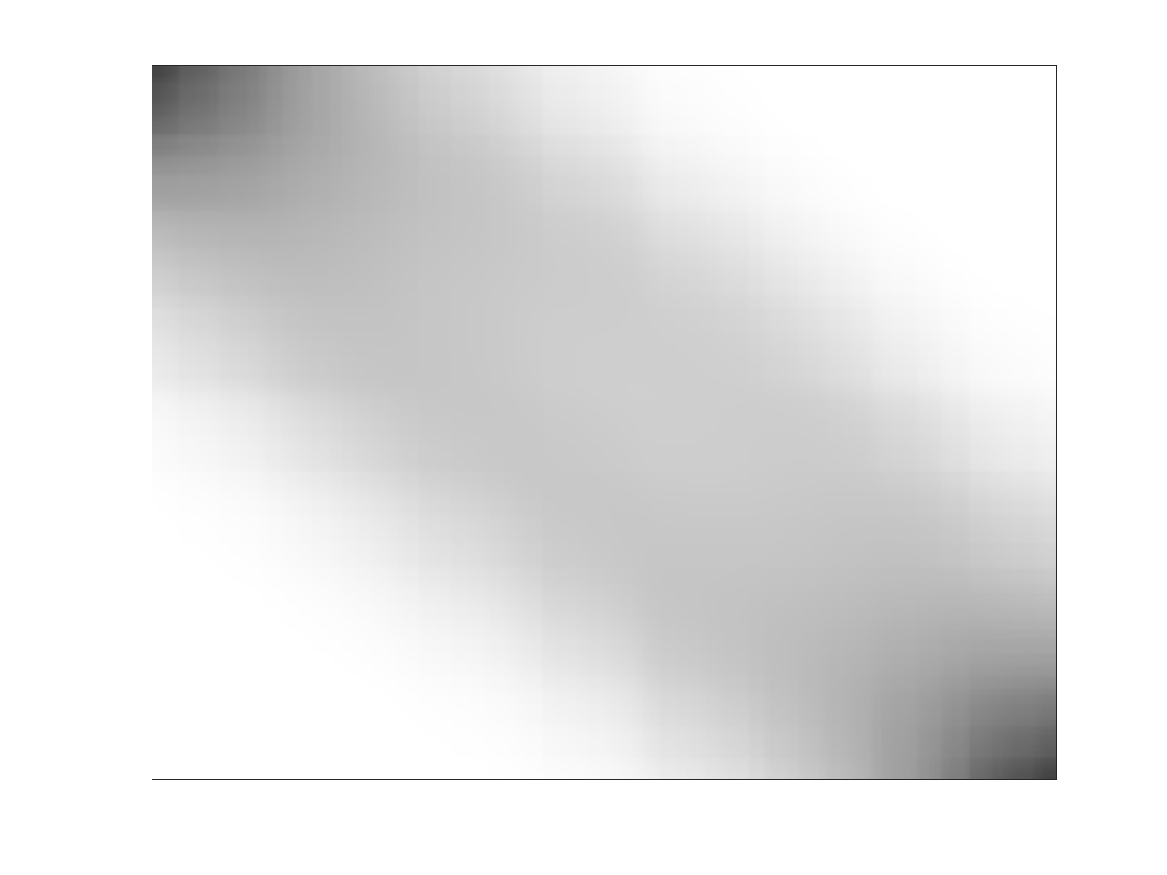}}
		\end{minipage}
		\begin{minipage}[c]{0.17\linewidth}
			\centering
			{${{t_6=1}}$\par\medskip
				\includegraphics[width=\textwidth]{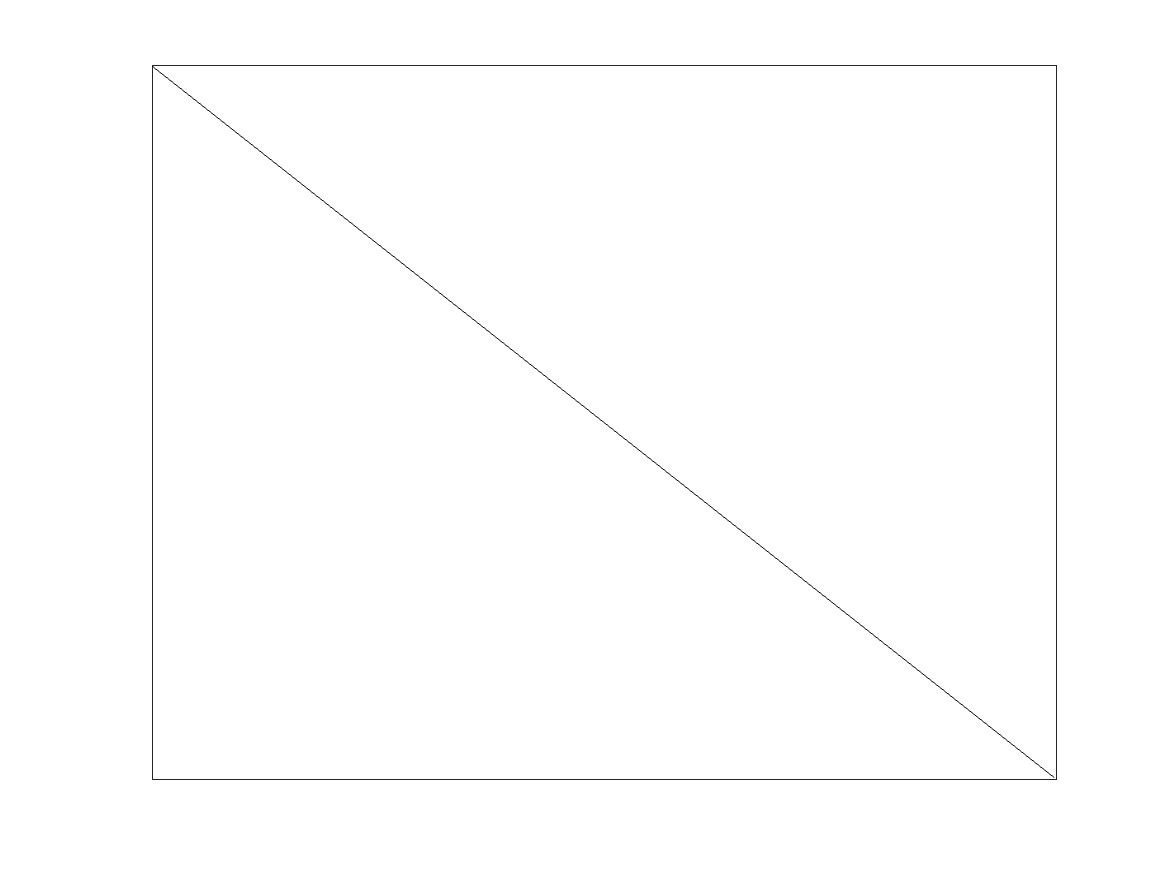}}
		\end{minipage}
		\caption{Joint measures $\boldsymbol{\Pi}_{1,k}$, $k=2, \cdots, 5$ representing the movement of the particles from initial position $x \in [0, 1]$ ($x$-axis) to position $x_{t_k}\in [0, 1]$ ($y$-axis) at time~$t_k$, where $\sigma(x)=1-x$.   \textbf{First row:} Sinkhorn algorithm. \textbf{Second row:} NFFT-Sinkhorn algorithm.}
		\label{fig:Euler-flow-Sinkhorn-map1}
	\end{figure}

	\begin{figure}[!ht]
		\centering
		\begin{minipage}[c]{0.17\linewidth}
			\centering
			{${{t_1=0}}$\par\medskip
				\includegraphics[width=\textwidth]{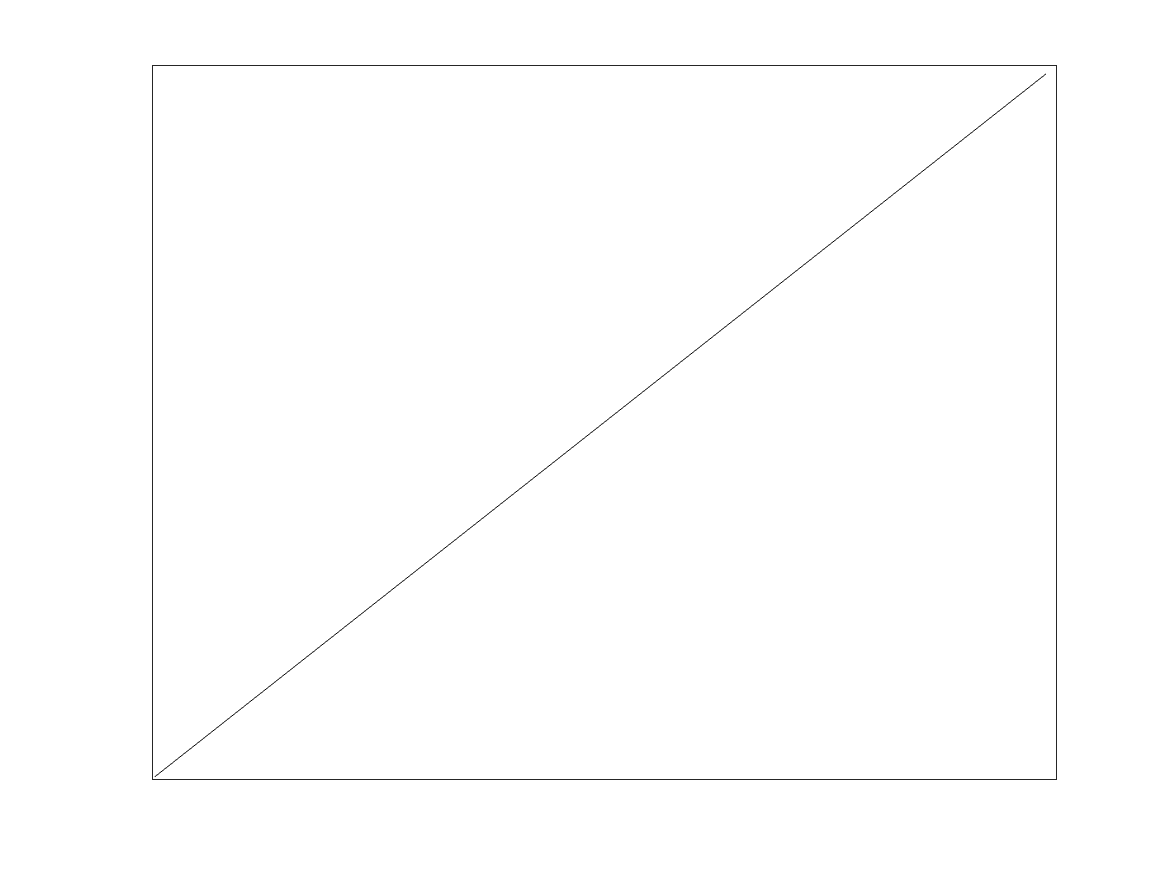}}\quad
		\end{minipage}
		\begin{minipage}[c]{0.15\textwidth}
			\centering
			{${{t_2=1/5}}$\par\medskip
				\includegraphics[width=\textwidth]{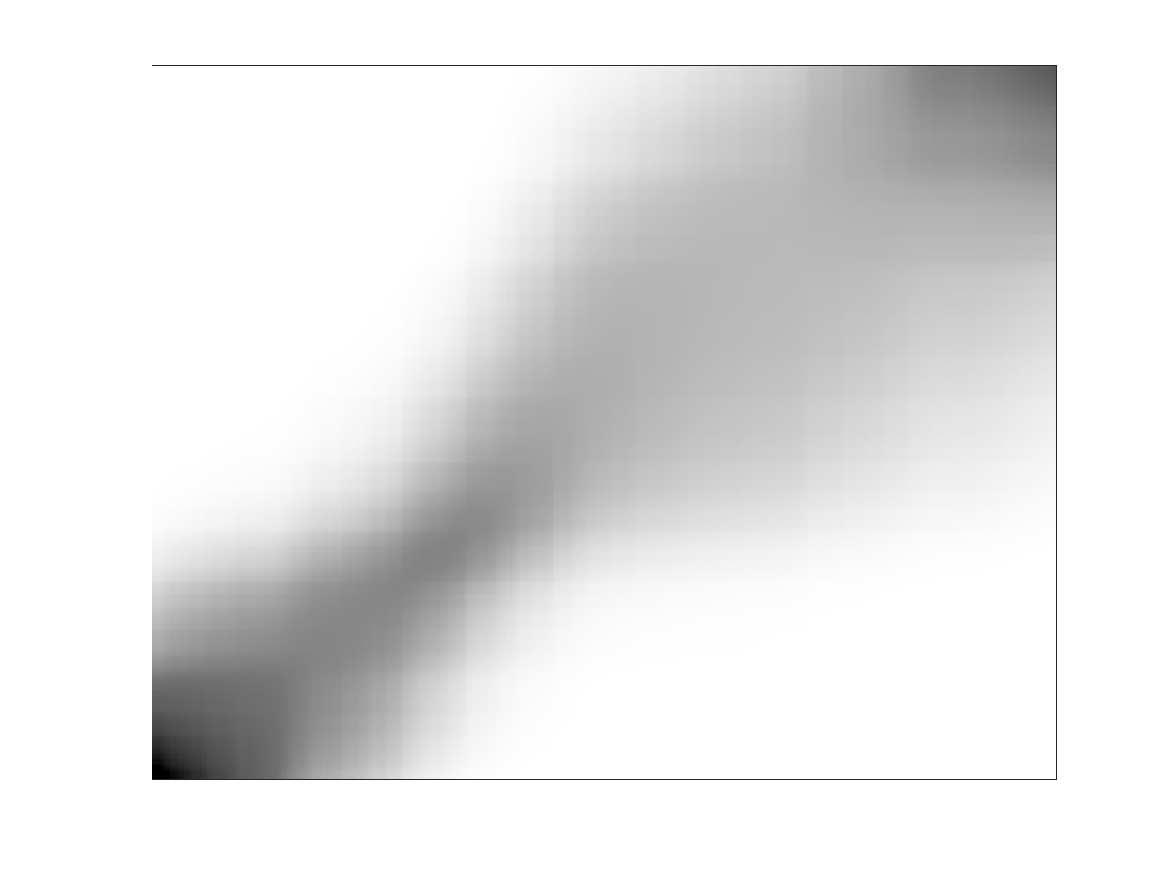}}\quad
			{\includegraphics[width=\textwidth]{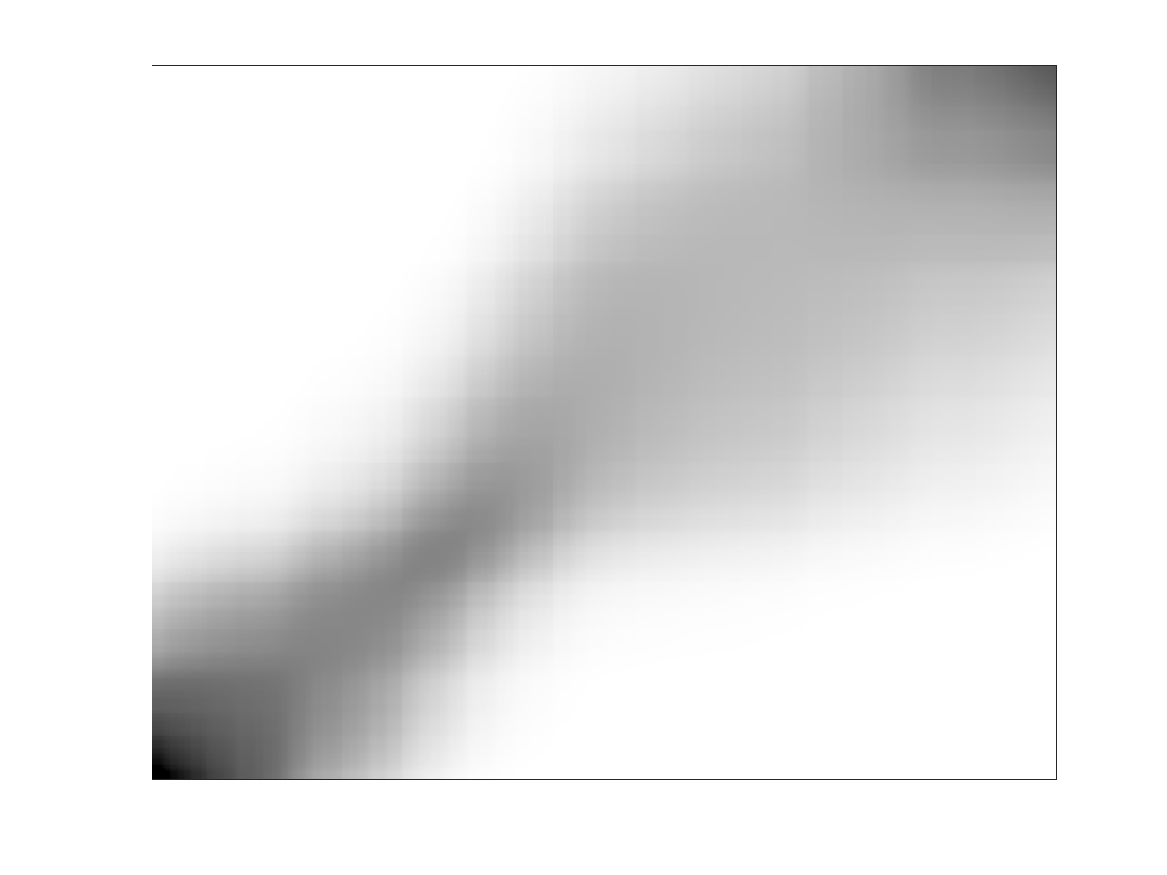}}
		\end{minipage}
		\begin{minipage}[c]{0.15\linewidth}
			\centering
			{${{t_3=2/5}}$\par\medskip
				\includegraphics[width=\textwidth]{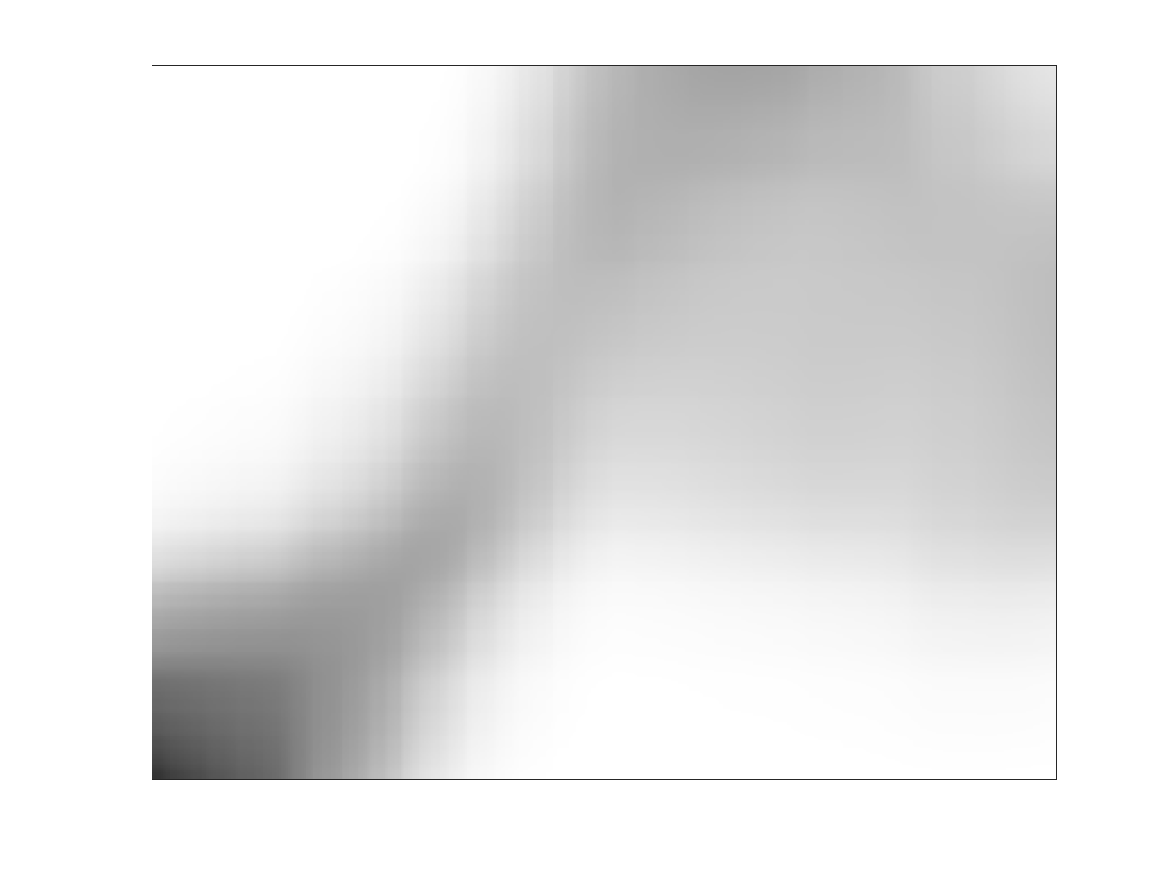}}\quad
			{
				\includegraphics[width=\textwidth]{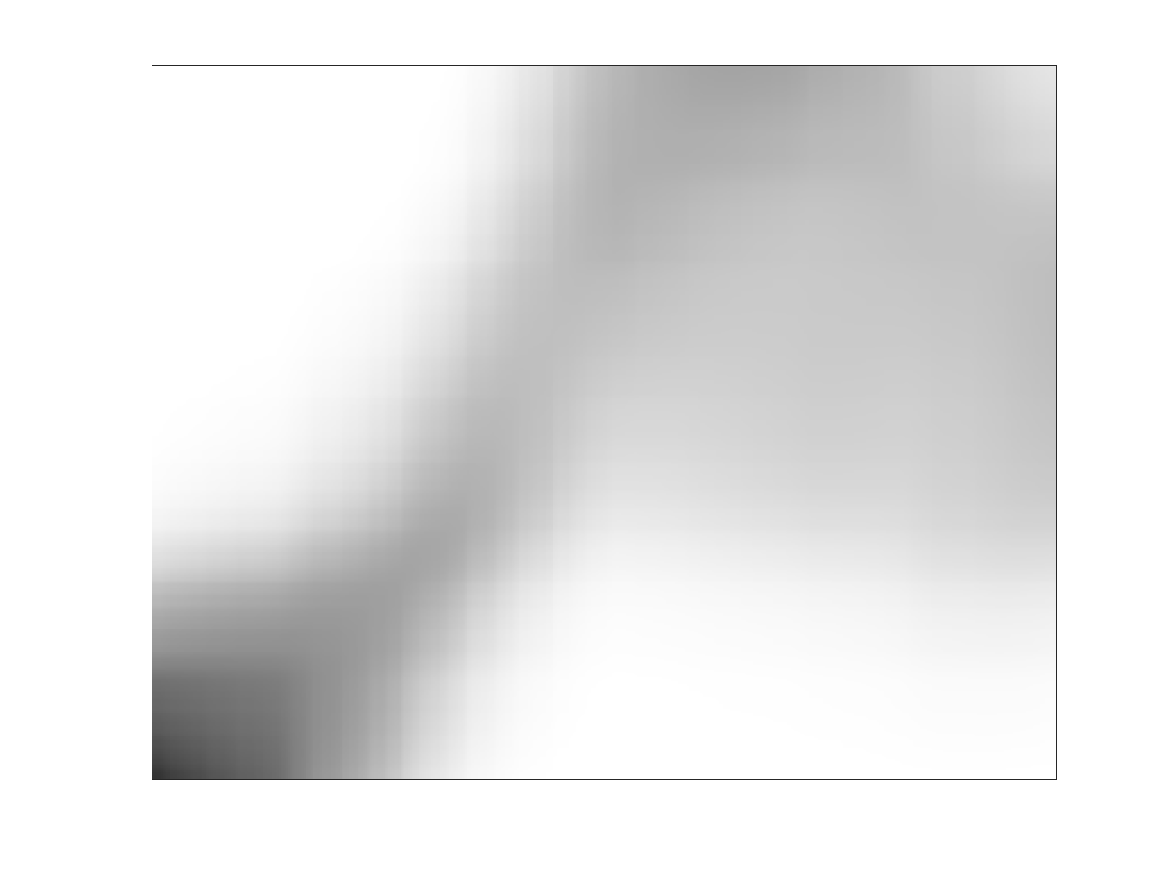}}
		\end{minipage}
		\begin{minipage}[c]{0.15\linewidth}
			\centering
			{${{t_4=3/5}}$\par\medskip
				\includegraphics[width=\textwidth]{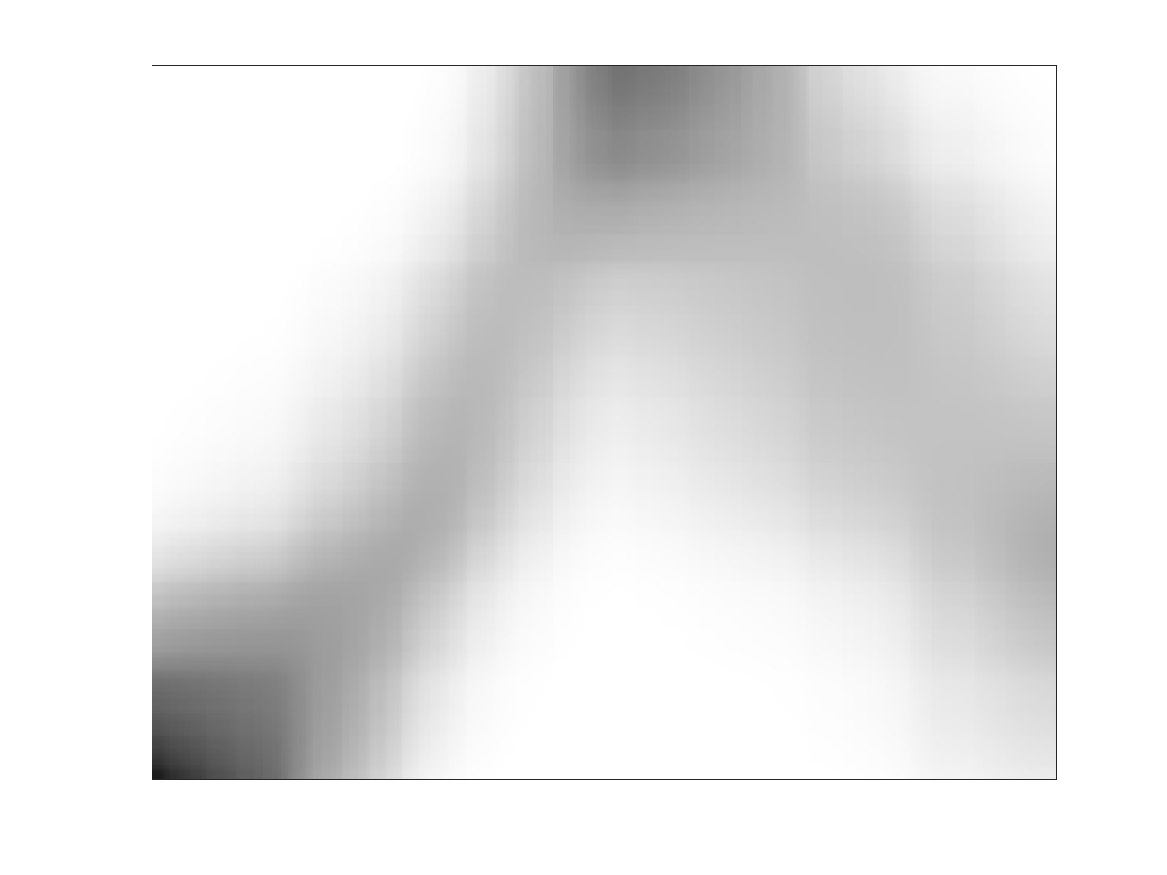}}\quad
			{\includegraphics[width=\textwidth]{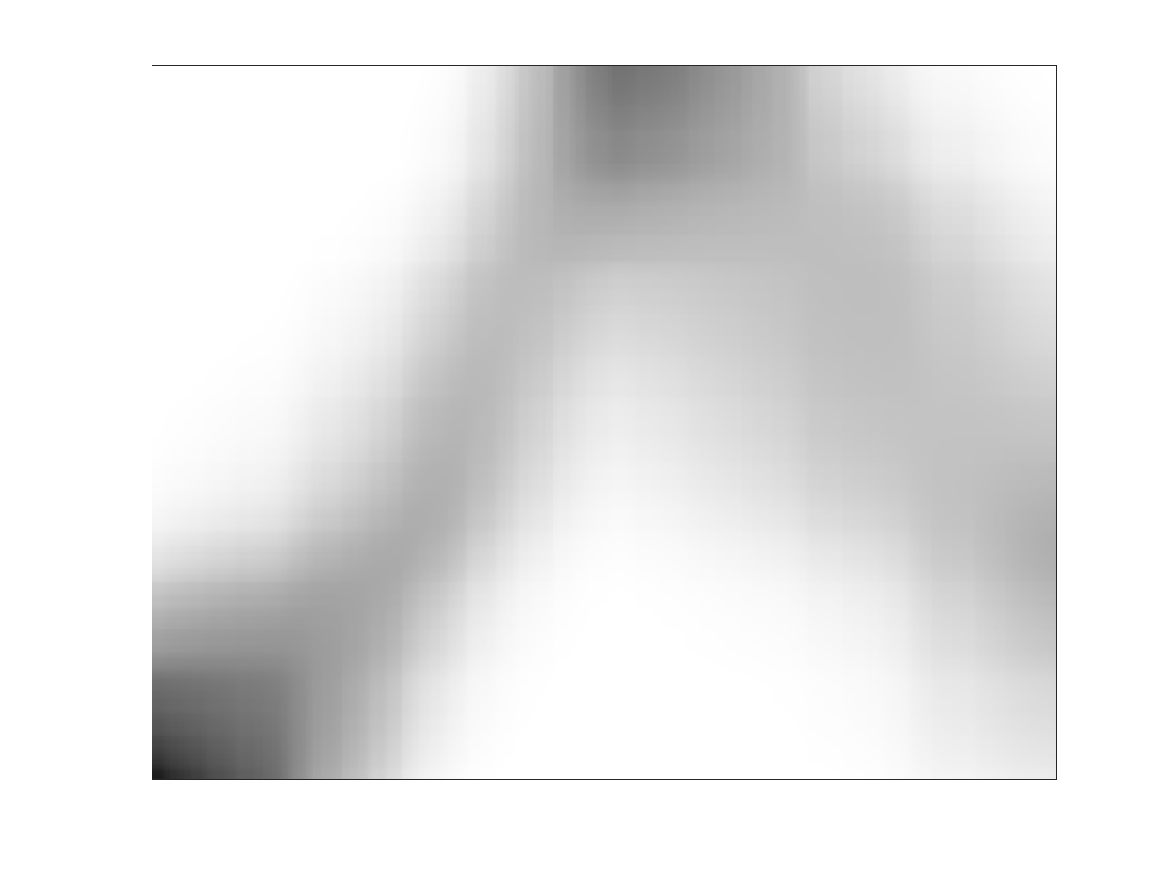}}
		\end{minipage}
		\begin{minipage}[c]{0.15\linewidth}
			\centering
			{${{t_5=4/5}}$\par\medskip
				\includegraphics[width=\textwidth]{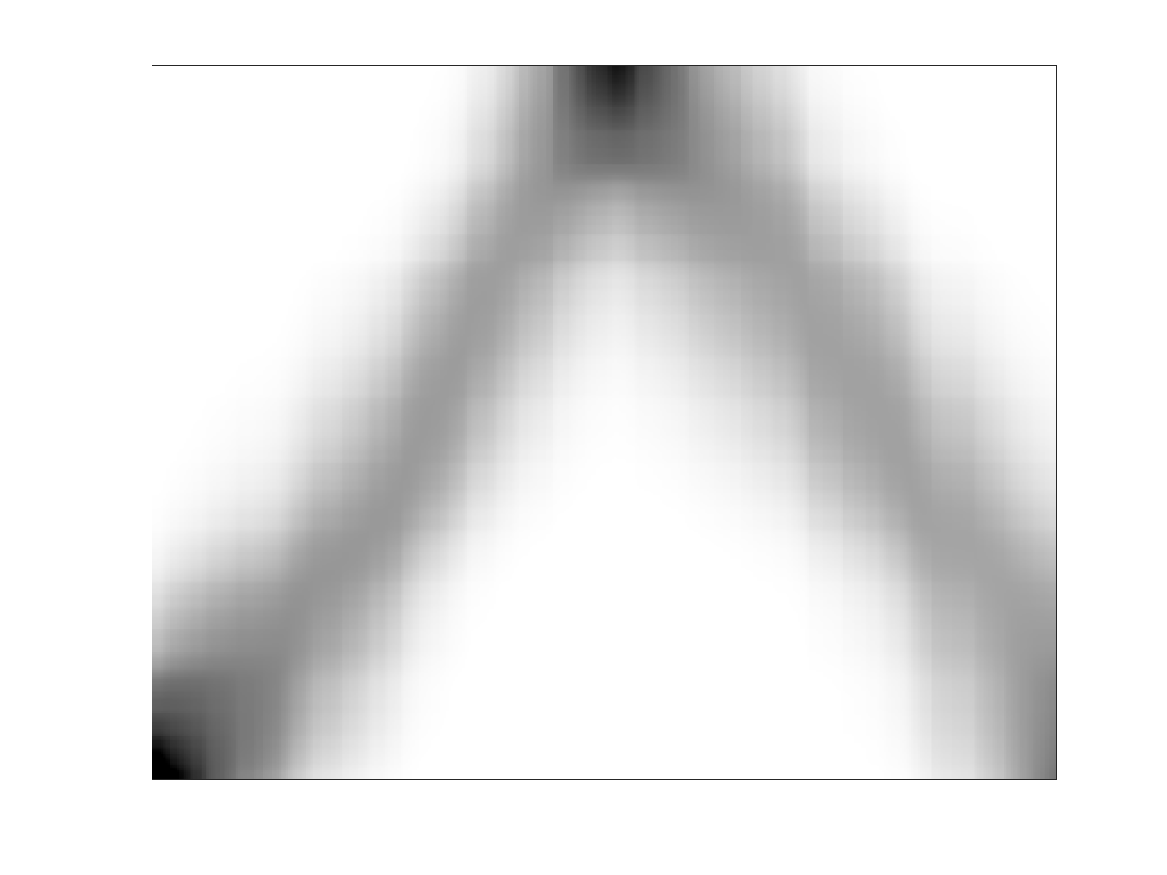}}\quad
			{\includegraphics[width=\textwidth]{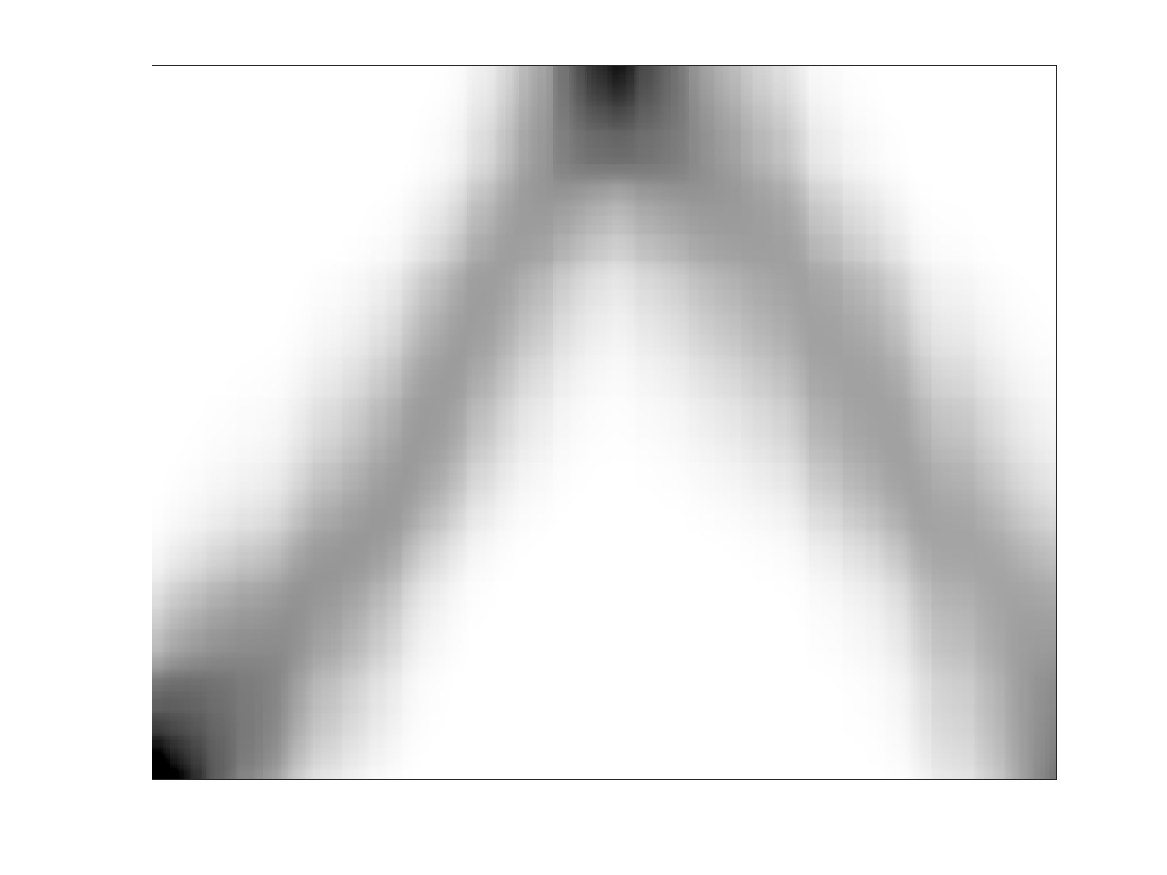}}
		\end{minipage}
		\begin{minipage}[c]{0.17\linewidth}
			\centering
			{${{t_6=1}}$\par\medskip
				\includegraphics[width=\textwidth]{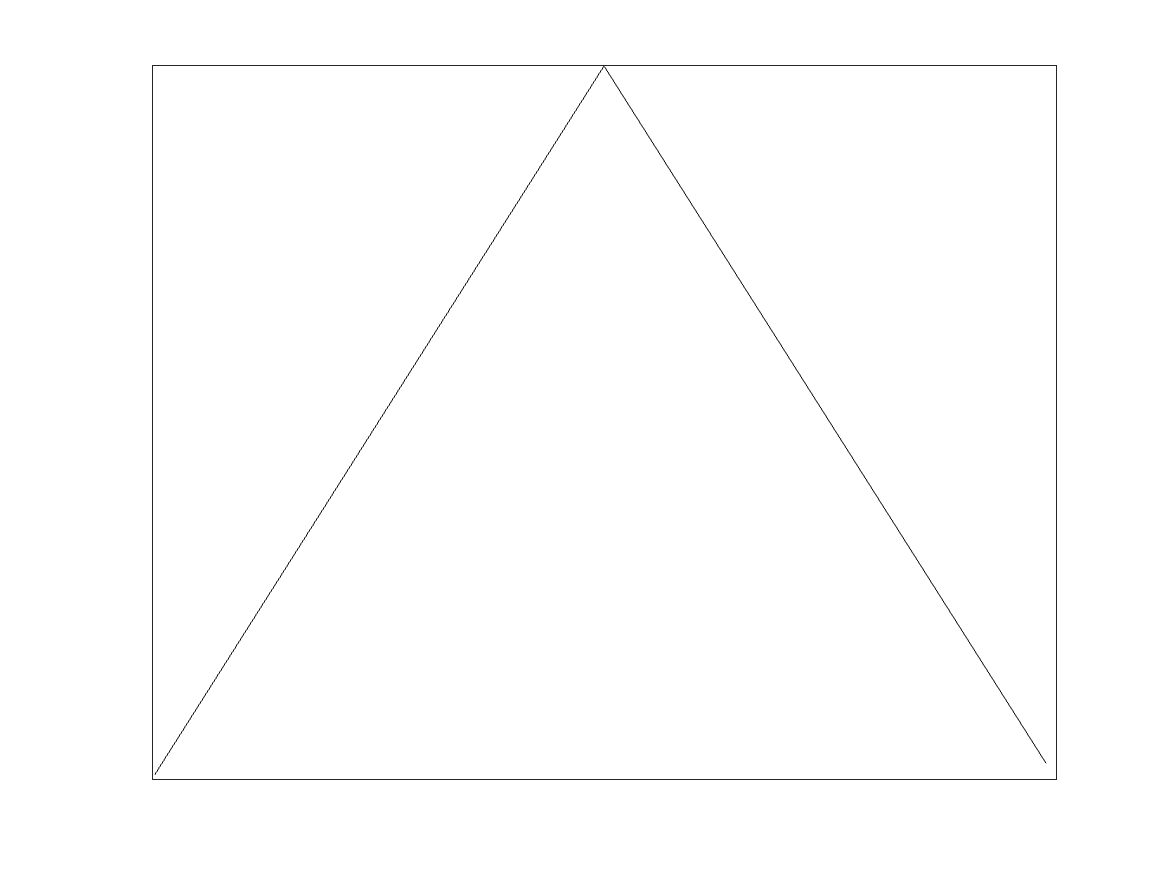}}
		\end{minipage}
		\caption{Joint measures as in Figure~\ref{fig:Euler-flow-Sinkhorn-map1}, but with the function $\sigma(x)=\min(2x, 1-2x)$. }
		\label{fig:Euler-flow-Sinkhorn-map2}
	\end{figure}
	
	\section{Conclusions} 
	We have proposed the NFFT-Sinkhorn algorithm to solve the $\MOTe$ problem efficiently. Assuming that the cost function of the multi-marginal optimal transport decouples according to a tree or a circle, we obtain a linear complexity in $K$. The complexity of the algorithm with respect to the numbers $n_k,$ $k\in [K]$, of atoms of the discrete marginal measures is further improved by the non-uniform fast Fourier transform. 
	This results in a considerable acceleration in our numerical experiments compared to the usual Sinkhorn algorithm. The tree-structured $\MOTe$ problem gives a much better numerical complexity than the circle-structured $\MOTe$ problem due to the fact that 
	in the latter case, matrix--matrix products are required in Algorithm~\ref{alg:sink-circle} instead of just matrix--vector products of Algorithm~\ref{alg:sink-tree}.
	
	\vspace{6pt}

	\subsection*{Acknowledgments}
	We gratefully acknowledge funding by the  BMBF $01$|S$20053$B project SA$\ell$E.
	Furthermore, we gratefully acknowledge funding by the German Research Foundation DFG (STE 571/19-1, project number 495365311). We also thank the anonymous reviewers for making valuable suggestions to improve the article.

\end{document}